\documentclass[twocolumn]{autart}

\usepackage{cite}
\usepackage{amsmath,amssymb,amsfonts,bm}
\usepackage{bbm}
\usepackage{graphicx}
\usepackage{subfigure}
\usepackage{textcomp}
\usepackage{algorithm}
\usepackage{algpseudocode}
\usepackage{enumitem}
\usepackage{mathrsfs}
\usepackage{xcolor}

\makeatletter
\renewcommand*\env@matrix[1][*\c@MaxMatrixCols c]{%
	\hskip -\arraycolsep
	\let\@ifnextchar\new@ifnextchar
	\array{#1}}
\makeatother

\allowdisplaybreaks[4]

\newenvironment{proof}{\begin{pf}}{\end{pf}}

\makeatletter
\def\add@ead#1#2{%
  \expandafter\add@tok\csname #1@text\endcsname{\texttt{#2}\ead@au}%
  \ignorespaces
}
\makeatother
\begin{document}

\begin{frontmatter}

\title{Adaptive Linear Quadratic Control of Unknown Linear Time-Varying Systems via Policy Gradient Methods\thanksref{footnoteinfo}}

\thanks[footnoteinfo]{This work was supported by ETH Zurich and the SNF through the NCCR Automation. {A preliminary version of part of this work has been accepted for presentation at the 2026 European Control Conference (ECC).} Corresponding author F.~Zhao.}

\author[eth]{Feiran Zhao}\ead{zhaofe@control.ee.ethz.ch},
\author[eth]{Florian D{\"o}rfler}\ead{dorfler@control.ee.ethz.ch}

\address[eth]{Department of Information Technology and Electrical Engineering, ETH Z{\"u}rich, 8092 Z{\"u}rich, Switzerland.}

\begin{keyword}
Adaptive control; linear quadratic regulator; linear systems; data-driven control.
\end{keyword}

\begin{abstract}
Unknown linear time-varying (LTV) systems require the control policy to adapt from online closed-loop data as dynamics evolve. Existing methods usually update the policy by solving a one-shot optimization problem, which can be computationally demanding and sensitive to noisy model estimates. In this paper, we propose a policy gradient adaptive control (PGAC) method for LTV system control with unknown model parameters. Specifically, PGAC integrates online policy optimization into feedback by updating the state-feedback policy with one-step gradient descent of the linear quadratic regulator cost at each time instant. This incremental update is computationally light and naturally limits policy variations caused by noisy data. To explicitly compute the policy gradient online, we estimate local models from recent closed-loop trajectories using normalized sliding-window least-squares. We provide stability and convergence certificates of PGAC for two classes of LTV systems. For slowly time-varying systems, we prove that the closed-loop system achieves practical exponential stability without a dwell-time condition. For piecewise-constant LTV systems, we establish practical stability through a dwell-time contraction argument. We also provide average frozen-time optimality-gap bounds of the policy sequence for both classes. Finally, we validate the effectiveness of our method via numerical case studies of both LTV and nonlinear systems.
\end{abstract}

\end{frontmatter}

\section{Introduction}
Recent years have witnessed the resurgence of data-driven control, which aims to learn controllers for unknown dynamical systems from raw data matrices. {Many} approaches can be broadly divided into \textit{indirect} (system identification followed by model-based control) and \textit{direct} (bypassing system identification) approaches, which have their own pros and cons \cite{dorfler2023data}. Representative instances include data-driven predictive control \cite{coulson2019data, breschi_data-driven_2023, berberich2020data}, data-driven linear quadratic regulator (LQR) \cite{de2019formulas, dorfler2021certainty}, the data-informativity approach \cite{van2020data}, and regret-minimization control \cite{tsiamis2023statistical}. Most of them focus on offline data and linear time-invariant (LTI) systems, which serve as a primary benchmark for comparing and validating different data-driven control methods.

{Linear time-varying (LTV) systems, whose model parameters are unknown and change over time, are an expressive system class motivated by practical applications \cite{rugh2000research, falcone2007predictive}. Since future parameters are usually unpredictable, controlling LTV systems requires continual adaptation of the feedback policy from online closed-loop data. This requirement fundamentally differs from offline data-driven control: the controller must be efficiently recomputed or updated while the plant evolves, and the data window may contain samples generated by different dynamics. The aforementioned methods \cite{coulson2019data, breschi_data-driven_2023, berberich2020data, de2019formulas, dorfler2021certainty, van2020data}, however, learn a policy from a batch of offline data, and this policy is fixed during the operation. Hence, they cannot be directly adopted for LTV system control.}

{To this end, LTV variants of data-driven and adaptive control methods have been proposed recently \cite{nortmann2023direct, rotulo2022online, liu2023online, iannelli2025hybrid, qu2021stable, minasyan2021online, dai2026online, bartos2025stability}. The LQR parameterization \cite{de2019formulas} is extended to input-affine nonlinear systems \cite{dai2026online} and switched linear systems \cite{rotulo2022online} by using online collected data, where the policy is updated at every time step and the stability certificates rely on dwell-time conditions \cite{rotulo2022online}. The data-informativity approach \cite{van2020data} is extended to slowly time-varying systems in \cite{liu2023online}, leading to stabilizing policies by solving a sequence of data-based semi-definite programming (SDP) problems online. A direct adaptive control method is proposed in \cite{iannelli2025hybrid}, where the linear state-feedback policy is updated via an event-triggered rule of a data-based Lyapunov function. The indirect certainty-equivalence LQR approach is adopted to stabilize LTV systems with a projected least mean square (LMS) estimator \cite{bartos2025stability}. There are also regret-minimization control methods that aim at both stabilization and low incurred cost of LTV systems \cite{qu2021stable, minasyan2021online}. These approaches provide important stability and performance guarantees, but many of them repeatedly solve optimization problems such as SDPs or Riccati equations \cite{qu2021stable, rotulo2022online}. For high-dimensional systems or fast variations, such one-shot online optimization can be computationally demanding; moreover, the optimizer computed from noisy closed-loop data may vary significantly from one update to the next, which can lead to unfavorable robustness and stability issues.}

{The policy gradient (PG) method, an essential approach of reinforcement learning \cite{bin2022towards, fazel2018global}, provides a lightweight alternative to one-shot online optimization. Instead of recomputing a full optimizer, PG updates the current policy by one descent step of the LQR cost. For the LQR, PG methods enjoy global convergence guarantees under proper initialization and stepsize choices \cite{fazel2018global}, and the gradient has a closed-form expression once a model estimate is available. This incremental structure is particularly appealing for time-varying adaptive control: the gain update is computationally light and can be regulated by the stepsize, which improves robustness to noisy model estimates compared with one-shot certainty-equivalence optimization \cite{qu2021stable, rotulo2022online}. Our previous works \cite{zhao2024data, zhao2025policy} developed policy gradient adaptive control (PGAC) for unknown LTI systems, where the policy is updated in closed-loop feedback using online data. PGAC has also seen successful applications on complicated nonlinear and time-varying systems, including district heating systems \cite{yi2026data}, grid-connected inverters \cite{11366674, zhao2024direct}, aerospace control~\cite{10981973}, and autonomous bicycle control~\cite{persson2025adaptive}.}

{However, it remains unclear if the policy gradient update has convergence and stability certificates for LTV systems~\cite{zhao2025policy}. Unlike standard PG methods for a fixed LQR objective \cite{fazel2018global, bin2022towards}, the objective here changes with the plant, and the gradient is computed from a sliding window of closed-loop data generated by time-varying dynamics. Thus, the policy, the model estimate, and the closed-loop state evolve simultaneously in feedback. In a preliminary conference version \cite{laurent2026adaptive}, PGAC was extended to switched linear systems with practical stability certificates; however, a dwell-time condition was indispensable even for slow variations, and the bounds were not uniform. This paper takes a step further and develops a stability and convergence theory for PGAC under unknown LTV dynamics.}

{Our main contributions are summarized as below.}
\begin{itemize}[leftmargin=*]
	\item {We develop a PGAC method for unknown LTV systems, where the policy is updated in closed-loop feedback from online data. Unlike existing LTV data-driven methods that repeatedly solve SDPs or Riccati equations \cite{liu2023online, bartos2025stability}, the proposed controller takes only one policy gradient step at each time instant. This first-order update makes the adaptation computationally light and lets the stepsize directly control the policy variation, which improves robustness to noisy model estimates compared with one-shot certainty-equivalence updates \cite{qu2021stable, rotulo2022online}.}
	\item {We propose normalized sliding-window least-squares to estimate the local model. This differs from the ordinary least-squares estimator used in the LTI PGAC framework \cite{zhao2025policy}: normalization removes the need for an a priori state bound in the identification error, while the sliding window keeps the estimate adaptive to recent dynamics. The error bounds explicitly separate the effects of temporal variation, process noise, and persistent excitation.}
	\item {For slowly time-varying systems, we prove that the PGAC policy sequence remains sequentially stable~\cite{cohen2019learning}. This yields practical exponential stability with a residual term determined by the probing signal and process noise, and it does not require the dwell-time conditions used in switched-system data-driven control \cite{rotulo2022online, laurent2026adaptive}. We also bound the average frozen-time optimality gap, which {consists of a transient term due to the initial policy and bias terms due to system variation and noisy finite-window data}.}
	\item {For piecewise-constant LTV systems, where abrupt jumps can destroy global sequential stability \cite{cohen2019learning}, we prove a two-layer stability certificate: the state decays inside each fixed mode, and the interval-wise maximum contracts across switches under a dwell-time condition. Compared with the preliminary switched PGAC result \cite{laurent2026adaptive}, the result allows infinitely many switches and keeps the constants uniform. Compared to the slowly varying case, the corresponding optimality gap bound additionally depends on the switching frequency.}
\end{itemize}

The remainder of this paper is organized as follows. Section II provides preliminaries. Section III formulates the stabilization problem of LTV systems and introduces the PGAC method. Section IV provides stability certificates for the proposed method. Section V uses simulations to validate the theoretical results. Conclusions are drawn in Section VI. All proofs are deferred to the Appendix.


\section{Preliminaries}
In this section, we introduce the preliminaries on the linear quadratic regulator (LQR) problem and the policy gradient method for the LQR.  

\subsection{The model-based LQR}
Consider a linear time-invariant (LTI) system
\begin{equation}\label{equ:sys}
	\left\{\begin{aligned}
		x_{t+1} & =A x_t+B u_t+w_t \\
		z_t & =\begin{bmatrix}
			Q^{1 / 2} & 0 \\
			0 & R^{1 / 2}
		\end{bmatrix}
		\begin{bmatrix}
			x_t \\
			u_t
		\end{bmatrix}
	\end{aligned}\right.,
\end{equation}
where $t\in \mathbb{N}$, $x_t\in\mathbb{R}^{n}$ is the state, $u_t\in\mathbb{R}^{m}$ is the control input, $w_t \in \mathbb{R}^n$ is the process noise, and $z_t$ is the performance signal of interest. We assume that {the pair} $(A,B)$ {is} controllable and the weighting matrices $(Q, R)$ are positive definite~\cite{bin2022towards}.  

The LQR problem is finding a state-feedback gain $K\in \mathbb{R}^{m\times n}$ to minimize the $\mathcal{H}_2$-norm of the transfer function $\mathscr{T}(K):w \rightarrow z$ of the closed-loop system
\begin{equation}\label{equ:closedsys}
	\begin{bmatrix}
		x_{t+1} \\
		z_t
	\end{bmatrix}=\begin{bmatrix}[c|c]
		A+BK & I_n \\
		\hline \begin{bmatrix}
			Q^{1 / 2} \\
			R^{1 / 2} K
		\end{bmatrix} & 0
	\end{bmatrix}\begin{bmatrix}
		x_t \\
		w_t
	\end{bmatrix}.
\end{equation}
When $A+BK$ is stable, it holds that \cite{anderson2007optimal}
\begin{equation}\label{equ:transfer}
	\|\mathscr{T}(K)\|_2^2  = \text{Tr}((Q+K^{\top}RK)\Sigma)=:C(K),
\end{equation}
where $\Sigma$ is the closed-loop state covariance matrix obtained as the positive definite solution to the Lyapunov equation
\begin{equation}\label{equ:Sigma}
	\Sigma = I_n + \left(A+BK\right)\Sigma \left(A+BK\right)^{\top}, 
\end{equation}
and we refer to $C(K)$ as the nominal LQR cost.

Given model parameters $(A,B)$, the LQR problem (\ref{equ:transfer})-(\ref{equ:Sigma}) can be solved analytically by {standard model-based methods}, e.g., via the algebraic Riccati equation~\cite{anderson2007optimal}. In the sequel, we recapitulate the policy gradient method stemming from reinforcement learning (RL) to  iteratively solve the LQR problem (\ref{equ:transfer})-(\ref{equ:Sigma}).

\subsection{Policy gradient methods for the LQR}
The policy gradient method directly updates the feedback gain towards the optimal LQR gain using gradient descent \cite{bin2022towards}
\begin{equation}\label{equ:grad}
	K^+ = K - \eta \nabla C(K),
\end{equation}
where $\nabla C(K)$ is the gradient of the LQR with respect to the policy $K$.  Define $\mathcal{S}:=\{K\in \mathbb{R}^{m\times n}|\rho(A+BK)<1\}$ as the set of stabilizing gains. Then, for any $K\in \mathcal{S}$, the gradient $\nabla C(K)$ has a closed-form expression.
\begin{lemma}[{\cite[Lemma 1]{fazel2018global}}]\label{lem:mb_pg}
	For any $K\in\mathcal{S}$, the gradient of $C(K)$ is given by
	\begin{equation}\label{equ:gradK}
		\nabla C(K) =  2\left(\left(R+{B}^{\top} P {B}\right) K+{B}^{\top} P {A}\right) {\Sigma},
	\end{equation} 
	where ${\Sigma}$ satisfies \eqref{equ:Sigma}, and $P$ is the unique positive definite solution to the Lyapunov equation 
	\begin{equation}\label{equ:Lyap}
		P = Q + K^{\top}RK + ({A}+{B}K)^{\top}P ({A}+{B}K).
	\end{equation}
\end{lemma}

By Lemma \ref{lem:mb_pg}, computation of the gradient requires model parameters $(A,B)$ and the solution of two Lyapunov equations \eqref{equ:Sigma} and \eqref{equ:Lyap}. While the {nominal} LQR cost $C(K)$ is non-convex in the feedback gain $K$, it satisfies a gradient dominance property, leading to linear convergence of the policy update \eqref{equ:grad} to $K^*$ under a proper stepsize and an initial stabilizing gain~\cite{fazel2018global}. 
 

Next, we shift our attention from the LTI system \eqref{equ:sys} to linear time-varying (LTV) systems with unknown and unpredictable model parameters, and use policy gradient methods for adaptive control design.

	\section{Problem formulations}
	This section formulates the {adaptive control} problem of two classes of time-varying systems: slowly time-varying systems with small continuous variations and piecewise-constant LTV systems.  
	\subsection{Adaptive control of slowly time-varying systems}
	We first consider a slowly time-varying system
  	\begin{equation}\label{equ:ltv_sys}
  		x_{t+1} = A_tx_t + B_tu_t + w_t,
  	\end{equation}
  	where the system matrices $A_t$ and $B_t$ are unknown, may vary at any time, and satisfy the following controllability assumption commonly adopted in the literature \cite{rotulo2022online, liu2023online}.
  	\begin{assum}[Frozen-time controllability]\label{assum:control}
  		For every $t\in \mathbb{N}$, the pair $(A_t, B_t)$ {is} controllable.
  	\end{assum}
  	
  	We also assume that $\{(A_t, B_t)\}$ do not drift to infinity, and the process noise $w_t$ is uniformly bounded.
  	\begin{assum}[Bounded dynamics and noise]\label{ass:bounded}
  		There exist constants $a_m$ and $b_m$ such that  $\|A_t\|\leq a_m$ and  $\|B_t\|\leq b_m, \forall t\in \mathbb{N}$. Moreover, there exists a constant $w_m\geq 0$ such that $\|w_t\|\leq w_m$ for all $t\in\mathbb{N}$.
  	\end{assum} 
  	
  	This assumption can be satisfied if $(A_t, B_t)$ belong to a compact set and is mild in practice.
  	
  	Denote the variation of system matrices when the dynamical model changes from $(A_{t-1}, B_{t-1})$ to $(A_t, B_t)$ as
  	\begin{equation}\label{equ:variation}
  		\Delta_t:= \begin{bmatrix}
  			B_{t-1}- B_t & A_{t-1}-A_t
  		\end{bmatrix}.
  	\end{equation} 
  	Then, we assume that the variation of $(A_t, B_t)$ has a uniform bound \cite{liu2023online}.
  	
  	\begin{assum}[Slow variation]\label{assum:variation}
  		There exists a constant $\delta > 0$ such that $\|\Delta_t\| \leq \delta, \forall t \in \mathbb{N}_+$.
  	\end{assum}

  	 
Compared with existing LTV control methods that typically assume known
time-varying models
\cite{qu2021stable,tanaskovic2019adaptive}, our setting is
fully data-driven and adaptive. In particular, our goal is to design a lightweight
online policy adaptation method that uses recent trajectory data to follow
model variations while maintaining stability of the closed-loop LTV system
\cite{liu2023online}. Since the presence of process noise and model variations
generally prevents convergence to the origin, we adopt the following standard
notion of practical exponential stability (PES), which is closely related to
practical stability and uniform ultimate boundedness in nonlinear control
\cite{khalil2002nonlinear,lakshmikantham1990practical,teel1995tools}.

\begin{definition}[Practical exponential stability]
	The LTV system~\eqref{equ:ltv_sys} is said to be practically exponentially
	stable if there exist constants $c>0$, $\rho\in(0,1)$, and $r\geq 0$ such
	that, for any $t\geq s\geq 0$,
	\[
	\|x_t\|\leq c\rho^{t-s}\|x_s\|+r.
	\] 
\end{definition}

\begin{problem}\label{problem: 1}
	Design an adaptive control algorithm for the unknown LTV system
	\eqref{equ:ltv_sys}, using only online closed-loop data, such that the
	resulting closed-loop system is PES and the controller tracks the
	frozen-time optimal LQR policy.
\end{problem}

  	\subsection{Adaptive control of piecewise-constant LTV systems}
  	
  	We next consider a second class of time-varying dynamics, namely the piecewise-constant LTV system
  	\begin{equation}\label{equ:switch}
  		x_{t+1} = A_i x_t + B_i u_t + w_t, \quad T_{i} \leq t < T_{i+1},~~ \text{for} ~i \in \mathbb{N}, 
  	\end{equation} 
  	where $(A_i, B_i)$ are unknown, and $\{T_i\}$ are the unknown switching instants when the dynamics jump from $(A_{i-1}, B_{i-1})$ to $(A_i, B_i)$. As in Assumption \ref{assum:variation}, we also assume that the jump is bounded.
  	\begin{assum}[{Bounded switching jumps}]\label{assum:switch_jump}
  		{There exists $\delta>0$ such that}
  		\[
  			{\|[B_i,A_i]-[B_{i-1},A_{i-1}]\|\leq \delta,\qquad \forall i\in\mathbb{N}_+.}
  		\]
  	\end{assum}
  	
  	
  	The reason  we regard the piecewise-constant system control as a different setting is that it introduces different challenges compared with
  	the slowly time-varying case in Section~3.1. Specifically at a switching instant, the
  	model variation  in Assumption \ref{assum:switch_jump}  can be larger, and hence the closed-loop matrices may change
  	abruptly, leading to {a completely different analysis}. For example, we usually require an additional {dwell-time} condition to ensure stability. 
  	
  	
  	\begin{definition}[dwell time]\label{def:dwell}
  		We refer to an integer $\tau$ as the dwell time, if for all $i\in \mathbb{N}$, it holds that $T_{i+1} - T_i \geq \tau$.
  	\end{definition}
 
 	The dwell time is used to let the state decrease enough after a switch. In this setting, it is usually difficult to prove exponential decay of the state with respect to {time steps}. Instead, we focus on PES with respect to the switching indices.
  	
  	\begin{definition}[Interval-wise PES]
  		The piecewise-constant system \eqref{equ:switch} is said to be interval-wise practically
  		exponentially stable if there exist constants $c>0$, $\rho\in(0,1)$,
  		and $r\geq 0$ such that
  		\[
  		\|x_{\max,i}\|\leq c\rho^i\|x_{T_0}\|+r,\qquad \forall i\in\mathbb{N},
  		\] 
  		where $\|x_{\max,i}\|:= \max_{T_i\leq t\leq T_{i+1}} \|x_t\|$.
  	\end{definition}
  	
  	Our problem formulation for the piecewise-constant LTV system \eqref{equ:switch} is as follows.
  	
  	\begin{problem}\label{problem: 2}
  		Design an adaptive control algorithm for \eqref{equ:switch}, using only
  		online closed-loop data, such that the closed-loop system achieves
  		interval-wise PES and the controller tracks the
  		frozen-time optimal LQR policy.
  	\end{problem}
  	
  	Next, we propose a unified adaptive control algorithm to solve Problems \ref{problem: 1} and \ref{problem: 2} by leveraging policy gradient methods.  
	
	\section{{Policy gradient adaptive control for LTV systems}}
	In this section, we first propose the policy gradient adaptive control algorithm for LTV systems. Then, we provide {stability certificates and average frozen-time optimality-gap bounds} of Algorithm \ref{algo2_switch} for slowly time-varying systems \eqref{equ:ltv_sys} and piecewise-constant LTV systems \eqref{equ:switch}, respectively.  
	
	\subsection{The policy gradient adaptive control method for LTV systems}

	\begin{figure}[t]
		\centerline{\includegraphics[width=65mm]{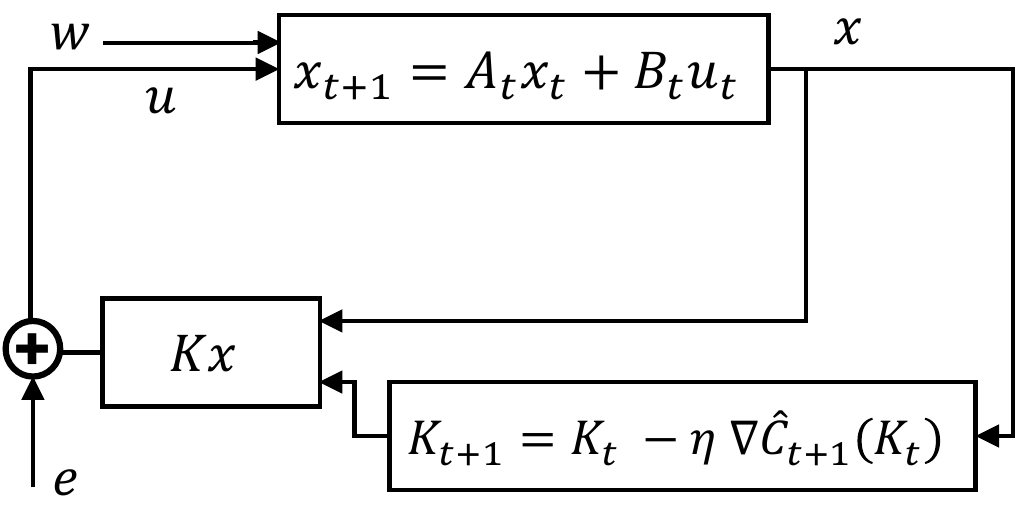}}
		\caption{Illustration of the policy gradient adaptive control algorithm.}
		\label{pic:diagram}
	\end{figure}
	
	\begin{algorithm}[t]
		\caption{Policy gradient adaptive control for linear time-varying systems with sliding window data}
		\label{algo2_switch}
		\begin{algorithmic}[1]
			\State \textbf{Initialize:} Offline data \((X_{t_0}, U_{t_0}, X_{t_0+1})\), {a policy \(K_{t_0}\) stabilizing the frozen-time pair \((A_{t_0},B_{t_0})\)}, and a stepsize \(\eta\).
			\For{ \(t = t_0, t_0 + 1, \ldots\) }
			\State Apply \(u_t = K_t x_t + e_t\), {and then} observe \(x_{t+1}\).
			\State Estimate a model using $(X_{t+1}, U_{t+1}, X_{t+2})$ and {normalized} least-squares
			\[
			{[\hat{B}_{t+1}, \hat{A}_{t+1}] = \mathop{\arg\min}_{B, A} \| \bar X_{t+2} - [B, A]\bar D_{t+1} \|_F.}
			\]
			\State Perform one-step policy gradient descent
			\begin{equation}\label{equ:mbgd}
				K_{t+1} = K_t - \eta \nabla \hat{C}_{t+1}(K_t),    
			\end{equation} 
			where \(\nabla \hat{C}_{t+1}(K_t)\) is the policy gradient with the estimated model \((\hat{A}_{t+1}, \hat{B}_{t+1})\).
			\EndFor
		\end{algorithmic}
	\end{algorithm}
	
	Our adaptive control algorithm is illustrated in Fig. \ref{pic:diagram} and detailed in Algorithm~\ref{algo2_switch}. We assume that the initial policy $K_{t_0}$ stabilizes the frozen-time pair $(A_{t_0},B_{t_0})$, which is standard in adaptive linear quadratic control literature~\cite{mania_certainty_2019,lu2023almost,simchowitz2020naive,cohen2019learning}.
	The control policy is in the form of linear state-feedback $u_t = K_t x_t + e_t$, where $e_t$ is the probing signal.
	Since the system is time-varying, at each time step we use only the most recent batch of online closed-loop data for system identification. Define the $L$-long sliding window data collected at time $t$
	\begin{equation}\label{equ:sliding_data}
		\begin{aligned}
			X_t &:= \begin{bmatrix}
				x_{t-L} & {x_{t-L+1}}& \dots& x_{t-1}
			\end{bmatrix}\in \mathbb{R}^{n\times L},\\
			U_t &:= \begin{bmatrix}
				u_{t-L} & {u_{t-L+1}}& \dots& u_{t-1}
			\end{bmatrix}\in \mathbb{R}^{m\times L},\\
			W_t &:= \begin{bmatrix}
				w_{t-L} & {w_{t-L+1}}& \dots& w_{t-1}
			\end{bmatrix}\in \mathbb{R}^{n\times L},\\
			X_{t+1} &:= \begin{bmatrix}
				x_{t-L+1} & {x_{t-L+2}}& \dots& x_{t}
			\end{bmatrix}\in \mathbb{R}^{n\times L},\\
			D_t &:= [U_t^{\top}, X_t^{\top}]^{\top}\in \mathbb{R}^{(m+n)\times L},
		\end{aligned}
	\end{equation}

	We estimate the current operating model using the normalized least-squares method. Define the normalized regression quantities
	\begin{align*}
		&d_s=[u_s^\top,x_s^\top]^\top,~
		n_s=\sqrt{1+\|d_s\|^2},~
		\bar d_s=\frac{d_s}{n_s}, \\
		&N_t:=\mathrm{diag}(n_{t-L},n_{t-L+1},\ldots,n_{t-1}),\\
		&\bar D_t:=D_tN_t^{-1},\qquad
		\bar X_{t+1}:=X_{t+1}N_t^{-1}.
	\end{align*}
	 
	Then, we identify a dynamical model that best fits the data matrices \eqref{equ:sliding_data} as
	\begin{equation}\label{equ:sysid}
		{[\hat{B}_t, \hat{A}_t] = \underset{B, A}{\arg \min }\left\|\bar X_{t+1}-[B,A]\bar D_t\right\|_F = \bar X_{t+1}\bar D_t^{\dagger}.}
	\end{equation}
		To ensure the uniqueness of the solution \eqref{equ:sysid}, we let the length of the input-state data satisfy $L\geq m+n$ and $D_t$ be persistently exciting (PE) by proper design of the probing signal. Specifically, the probing signal $e_t$ satisfies the following assumption. 
	\begin{assum}[Bounded and normalized PE]
		\label{assum:pe}
		The probing signal $e_t$ is chosen such that there {exist} constants $e_m>0$ and $\gamma>0$ such that
		\begin{equation}\label{equ:quant_pe}
			\underline{\sigma}\left(\frac{1}{L}\bar D_t\bar D_t^\top\right)
			\geq  \gamma^2,
			\qquad \forall t\ge t_0,
		\end{equation} 
		and
		\begin{equation}\label{equ:bund_prob}
			\|e_t\|\le e_m,\qquad \forall t\ge 0.
		\end{equation} 
	\end{assum}
	
	Assumption \ref{assum:pe} ensures a quantitative condition of persistency of excitation for our analysis and can be ensured by choosing bounded white noise. Note that the identification step in Algorithm \ref{algo2_switch} can be efficiently conducted by recursive rank-one update; we refer to \cite{zhao2025policy} for details.

	\begin{remark}[ {Why not ordinary least-squares}]\label{rem:nls}
		Our prior work \cite{zhao2025policy,zhao2024data} uses the ordinary least-squares method for model estimation, which cannot be adopted for LTV systems due to a circular dependency in adaptive control. Let us elaborate. It has been shown that the {estimation} error of ordinary least-squares for LTV systems depends on the condition number of the input-state matrix $D_t$ \cite{mareels1988persistency}. Thus, one needs a state bound to control $\|D_t\|$ to achieve a small identification error, and only afterwards can one prove the desired state bound. In comparison, normalization breaks this loop: the {estimation} error of normalized least-squares does not depend on $\|D_t\|$, and hence we adopt it in our method.
	\end{remark}

	Following the certainty-equivalence principle~\cite{dorfler2021certainty}, we treat $(\hat{A}_t, \hat{B}_t)$ estimated from windowed data as the ground-truth parameters at time $t$, and the corresponding LQR cost for the frozen-in-time system is
	\begin{equation}\label{prob:indirect} 
		\hat{C}_t(K) = \text{Tr}\left((Q+K^{\top}RK)\Sigma\right), 
	\end{equation}
	where $\Sigma$ is the unique positive definite solution to the following Lyapunov equation
	\begin{equation}\label{equ:indirect_sigma}
		\Sigma = I_n + \left(\hat{A}_t + \hat{B}_tK\right)\Sigma \left(\hat{A}_t + \hat{B}_tK\right)^{\top}.
	\end{equation} 	
	
	\begin{remark}[Certainty equivalence]\label{rem:ce}
		In this paper, certainty equivalence refers to the frozen-time LQR problem \eqref{prob:indirect}-\eqref{equ:indirect_sigma} built from the windowed least-squares estimate \eqref{equ:sysid}. Thus, the policy is updated toward the LQR gain of a local frozen-time surrogate, rather than toward a globally optimal controller for the full LTV trajectory. This surrogate is useful because the dynamics vary slowly within the window in \eqref{equ:ltv_sys}, or remain constant between switches in \eqref{equ:switch}, so the frozen-time LQR gradient provides a local descent direction for tracking the time-varying stabilizing policy.
	\end{remark}
	
	It is well-known in the LTI setting that the optimal LQR gain has guaranteed stability margins \cite{safonov1977gain}, and hence the optimal certainty-equivalence LQR gain of \eqref{prob:indirect}-\eqref{equ:indirect_sigma} may also stabilize the ground-truth system $(A_t, B_t)$. Motivated by this fact, we update the policy towards the certainty-equivalence LQR gain {in feedback} using online closed-loop data. Specifically, we 
	use the policy gradient method of the certainty-equivalence LQR cost \eqref{prob:indirect}-\eqref{equ:indirect_sigma} to update the policy, where the gradient is computed from \eqref{equ:gradK}--\eqref{equ:Lyap} with the true system matrices $(A,B)$ replaced by their windowed estimates $(\hat A_t,\hat B_t)$.


	\subsection{Certificates for slowly time-varying systems \eqref{equ:ltv_sys}} 
	
	A major challenge of theoretical analysis stems from the coupling between learning and control in Algorithm \ref{algo2_switch}: the controller is updated from online closed-loop data, while both the unknown system dynamics and the data used for identification vary over time. Moreover, inaccurate policy updates may destroy
	closed-loop stability, whereas overly conservative updates may fail to track the drift of the dynamics.

	We first quantify the model identification error of the normalized least-squares \eqref{equ:sysid} given data $(X_t, U_t, X_{t+1})$.
	\begin{lemma}\label{lem:id_error}
		Consider the {normalized} least-squares problem \eqref{equ:sysid} for the LTV system \eqref{equ:ltv_sys}, and let Assumptions \ref{assum:control}, \ref{ass:bounded}, \ref{assum:variation}, and \ref{assum:pe} hold.
		Then, for all $t\geq t_0$ it holds that
		\begin{equation}\label{equ:id_error}
			\|[\hat{B}_t, \hat{A}_t] - [B_t, A_t]\| \leq \frac{L\delta}{\gamma} + \frac{w_m}{\gamma}. 
		\end{equation} 
	\end{lemma}  
   
	 By Lemma \ref{lem:id_error}, the upper bound of the {estimation} error consists of two terms induced by the variation of the dynamics and process noise, respectively. In particular, the first one scales linearly with the variation of the dynamics $\delta$ and the window size $L$, which is lower bounded by $m+n$ for identifiability. The second {one} is linear in the uniform noise bound $w_m$. {Both terms are} inversely proportional to the normalized excitation level $\gamma$ in \eqref{equ:quant_pe}. 
	  
	 Next, we provide theoretical guarantees on the stability and frozen-time optimality. Since the closed-loop matrix $A_t + B_t K_t$ is time-varying, we adapt the notion of sequential stability from~\cite{cohen2018online} to the LTV setting.
	  
	 \begin{definition}[{Sequential stability}] \label{def:sss}
	 	A sequence of policies $\{K_t\}$ for the LTV system \eqref{equ:ltv_sys} is sequentially stable if there exist constants $\kappa \geq 1$, $0< \alpha \leq 1$, and matrices $\{H_t\succ 0\}$ and $\{L_t\}$ for $t \geq t_0$, such that $A_t+B_tK_t = H_tL_tH_t^{-1}$, and for all $t$, it holds that
	 	\begin{enumerate}[label = (\roman*)]
	 		\item $\|L_t\|\leq 1 - \alpha$ and $\|K_t\|\leq \kappa$;
	 		\item $\|H_t\|\leq \kappa$ and $\|H_t^{-1}\|\leq 1$;
	 		\item $\|H_{t+1}^{-1}H_t\| \leq 1 + \alpha/2$.
	 	\end{enumerate}
	 \end{definition}
	 
	By Definition \ref{def:sss}, if the closed-loop matrix $A_t + B_tK_t$ is frozen-time stable and changes slowly, then the system is sequentially stable, which further leads to practical exponential stability~\cite{cohen2018online}. The key challenge in analyzing Algorithm \ref{algo2_switch} is how to select the stepsize to ensure the aforementioned conditions and the improvement of the policy with perturbed gradient. 
	  
	 For notational simplicity, denote the identification error in \eqref{equ:id_error} by
	 \begin{equation}\label{equ:epsilon}
	 	{\epsilon:=\frac{L\delta+w_m}{\gamma}.}
	 \end{equation}
	  Then, our main results for the slowly time-varying system \eqref{equ:ltv_sys} are {as follows}.
	  
 	\begin{theorem}\label{thm:indirect}
 		Consider Algorithm \ref{algo2_switch} for the system \eqref{equ:ltv_sys}, and Assumptions \ref{assum:control}, \ref{ass:bounded}, \ref{assum:variation}, and \ref{assum:pe} hold. Then, there exist constants {$\nu_i>0,i\in\{1,2,\dots,9\}$} with $\nu_6 <1$ depending on $(a_m,b_m,Q,R,K_{t_0})$, such that, if 
 		$$
 		\delta \leq \nu_1,~ {\epsilon \leq \nu_2},~
 		\eta \leq \min\left\{\nu_3, {\frac{\nu_4}{\epsilon}}\right\}, 
 		$$ 
 		 then the sequence $\{K_t\}$ is sequentially stable, and the state is bounded as
		\begin{equation}\label{equ:bound_state}
			{\begin{aligned}
			\|x_t\| \leq{}& \nu_5\left(1-\frac{\nu_6}{2}\right)^{t-t_0}\|x_{t_0}\|\\
			&+\frac{2  \nu_5}{\nu_6}(b_me_m+w_m),
			\qquad \forall t\geq t_0 .
			\end{aligned}}
		\end{equation} 
		{Moreover, for any horizon $T>t_0$, the average frozen-time optimality gap satisfies}
		\begin{equation}\label{equ:regret_ltv}
			{
			\frac{1}{T-t_0}\sum_{t=t_0}^{T-1}\big(C_t(K_t)-C_t^*\big)
			\leq
			\frac{\nu_7C_{t_0}(K_{t_0})}{\eta(T-t_0)}
			+\frac{\nu_8\delta}{\eta}
			+\nu_9\epsilon .}
		\end{equation}
 	\end{theorem} 
 	
 	We make several remarks on Theorem \ref{thm:indirect}.
 	
 	First, the theorem requires the dynamics variation level $\delta$, the identification error $\epsilon$, and the stepsize $\eta$ to be sufficiently small. By \eqref{equ:epsilon}, the identification error consists of two components induced by the system variation and process noise, respectively. Therefore, the condition $\epsilon\leq \nu_2$ can be interpreted as a requirement that the variation of the dynamics and the noise level are sufficiently small relative to the excitation level $\gamma$. Moreover, the admissible stepsize decreases as the identification error increases through the condition $\eta \leq \nu_4/\epsilon$, reflecting the fact that larger model uncertainty requires more conservative policy updates.
 	
 	Second, under these conditions, the sequence of adaptive policies remains sequentially stable throughout the learning process. Consequently, the closed-loop system achieves practical exponential stability. In particular, \eqref{equ:bound_state} shows that the state trajectory consists of an exponentially decaying transient term and a residual term determined by the probing signal and process noise. Therefore, in the disturbance-free case with $e_m=w_m=0$, the state converges exponentially to the origin despite the continuous adaptation of both the controller and the model estimate.
 	
 	Third, \eqref{equ:regret_ltv} characterizes the frozen-time tracking performance of Algorithm~\ref{algo2_switch}. The average optimality gap is composed of three terms. The first term decreases as $\mathcal{O}(1/(\eta T))$ and represents the transient effect of the initial controller. The second term scales linearly with the variation level $\delta$, quantifying the intrinsic tracking error caused by the drift of the underlying dynamics. The third term is proportional to the identification error $\epsilon$ and captures the effect of imperfect model estimation. Consequently, when the system is time invariant and noise free, i.e., $\delta=\epsilon=0$, the average optimality gap converges to zero at the rate $\mathcal{O}(1/T)$.
 	
 	Finally, compared with existing adaptive control methods for LTV systems that repeatedly solve certainty-equivalent control problems based on estimated models, Algorithm~\ref{algo2_switch} updates the controller using lightweight first-order policy adaptation. As a result, the policy variation between consecutive updates can be directly controlled through the stepsize, which enables continuous closed-loop adaptation while preserving stability and tracking performance. {Moreover, Algorithm~\ref{algo2_switch} is computationally efficient in the sense that the policy update requires only computing a gradient. while computing the policy gradient still requires solving the Lyapunov equations in \eqref{equ:Sigma} and \eqref{equ:Lyap} with the estimated model, it is much cheaper than recomputing the certainty-equivalence minimizer through a Riccati equation or an SDP at every update as in \cite{qu2021stable, rotulo2022online}. }

\subsection{Certificates for piecewise-constant LTV systems \eqref{equ:switch}}

Compared with slowly varying systems \eqref{equ:ltv_sys}, piecewise-constant LTV systems \eqref{equ:switch} introduce additional challenges in both control and analysis. In the slowly varying case, the bounded variation assumption ensures that consecutive closed-loop matrices remain sufficiently close, which enables the sequential stability analysis and further leads to practical exponential stability in Theorem \ref{thm:indirect}. In contrast, abrupt parameter jumps in piecewise-constant LTV systems \eqref{equ:switch} may destroy this continuity property, causing the sequential stability condition to fail even when each frozen-time closed-loop system is individually stable. Consequently, the practical exponential stability guarantees established for slowly varying systems cannot be directly extended to the piecewise-constant setting.

As in Lemma \ref{lem:id_error}, we first quantify the model identification error for the system \eqref{equ:switch} with dwell time $\tau$. During the sliding window with length $L$, the maximal number of {switches} is $N = \lceil L/\tau \rceil$. Then, we have the following results.
 
\begin{lemma}\label{lem:id_error_sw}
	Consider the {normalized} least-squares problem \eqref{equ:sysid} {for the piecewise-constant LTV system \eqref{equ:switch}. Suppose that Assumptions \ref{assum:control}, \ref{ass:bounded}, \ref{assum:switch_jump}, and \ref{assum:pe} hold.} 
	Then, for $T_i \leq t < T_{i+1},~ \forall i\in \mathbb{N}$, it holds that
	\[
	\|[\hat{B}_t, \hat{A}_t] - [B_i, A_i]\|
	\leq \frac{N\delta}{\gamma}+\frac{w_m}{\gamma}.
	\]
\end{lemma} 

{The proof is given in the appendix.} Notice that the {estimation} error increases with $L$ but decreases with the dwell time $\tau$. In particular, if $\tau \geq L$, meaning that the sliding window data {are} from at most two different systems, then the {normalized} least-squares method achieves the minimal {estimation} error {$(\delta+w_m)/\gamma$}. {For notational simplicity, define}
\begin{equation}\label{equ:epsilon_sw}
	{\epsilon_{\rm sw}:=\frac{N\delta+w_m}{\gamma}.}
\end{equation}

Compared with the continuously varying system \eqref{equ:ltv_sys}, where the variation is assumed to be sufficiently small to preserve sequential stability, we allow a much larger variation $\delta$ in the piecewise-constant LTV system \eqref{equ:switch}. While this may violate sequential stability, we impose the dwell-time condition such that the state can sufficiently decay before the next switch. Let $i$ index the switching intervals $[T_i,T_{i+1})$. Denote $\|x_{\text{max}, i}\|:= \max_{T_i \leq t \leq T_{i+1}} \|x_t\|$ as the maximal state norm from $T_i$ to $T_{i+1}$ and $M_T:=\max\{i:T_i<T\}$ as the largest active switching-interval index up to $T$. For simplicity of presentation, let the initial time be $t_0 = T_0$.  
\begin{theorem}\label{thm:switch}
	Consider Algorithm \ref{algo2_switch} for the system \eqref{equ:switch}, and suppose that Assumptions \ref{assum:control}, \ref{ass:bounded}, \ref{assum:switch_jump}, and \ref{assum:pe} hold. Then, there exist constants {$\nu_i>0,i\in\{1,2,\dots,12\}$} depending on $(a_m,b_m,Q,R,K_{t_0})$ with $\nu_7\in(0,1)$, such that, if 
	\[
	{\begin{aligned}
	\tau &\geq \nu_1,\qquad
	\delta \leq \nu_2,\qquad
	\epsilon_{\rm sw}\leq \nu_3,\\
	\eta &\leq \min\left\{\nu_4,\frac{\nu_5}{\epsilon_{\rm sw}}\right\},
	\end{aligned}}
	\] 
	then {for all $i\in\mathbb{N}$ and all $T_i\leq t<T_{i+1}$, it holds that}
	\begin{equation}\label{equ:state_within_switch}
		\|x_t\| \leq \nu_6(1-\nu_7)^{t-T_i}\|x_{T_i}\|
		+\nu_8(b_me_m+w_m).
	\end{equation}
	{The interval maximum satisfies, for all $i\in\mathbb{N}$,}
	\begin{equation} 
		{\begin{aligned}
		\|x_{\text{max}, i}\| \leq{}& \nu_6\nu_9^i \|x_{T_0}\|\\
		&+ \left(1 + \frac{\nu_6}{1-\nu_9}\right)
		\nu_8(b_me_m+w_m).
		\end{aligned}}
	\end{equation} 
	{Here, $\nu_9=\nu_9(\tau)\in(0,1)$ decreases exponentially in $\tau$. Moreover, for any $T>t_0$, the average frozen-time optimality gap satisfies}
	\begin{equation}\label{equ:regret_sw}
		{\begin{aligned}
		&\frac{1}{T-T_0}
		\sum_{i=0}^{M_T}\sum_{t=T_i}^{\min\{T_{i+1},T\}-1}
		\big(C_i(K_t)-C_i^*\big)\\
		&\qquad\leq
		\frac{\nu_{10}C_{t_0}(K_{t_0})}{\eta(T-T_0)}
		+\frac{\nu_{11}M_T\delta}{\eta(T-T_0)}
		+\nu_{12}\epsilon_{\rm sw}.
		\end{aligned}}
	\end{equation}
\end{theorem} 

{Theorem \ref{thm:switch} differs from Theorem \ref{thm:indirect} in two essential aspects. First, the piecewise-constant setting permits abrupt jumps of the dynamics, so the sequential-stability argument used for slowly varying systems cannot hold globally in time. This is why a dwell-time condition is needed here: it gives the adaptive policy enough time to recover within each fixed mode before the next jump. The identification error $\epsilon_{\rm sw}$ also reflects this structure, since it depends on the number of switches contaminating the sliding window rather than on a per-step drift.}

{Second, the stability {result} has two layers. The bound \eqref{equ:state_within_switch} gives  a time-step decay of the state inside a fixed switching interval, while the interval-wise bound controls the maximal state norm across intervals and decays with the switching index. Thus, different from Theorem \ref{thm:indirect}, the result does not claim a single global exponential decay rate over all time steps; it combines within-mode decay with a dwell-time-based contraction across switches. The optimality-gap bound has the same interpretation: its drift term scales with the switching frequency $M_T/(T-T_0)$ instead of the continuous variation level in \eqref{equ:regret_ltv}. Compared to our previous work \cite{laurent2026adaptive}, the theorem allows infinitely many switches and keeps the constants uniform.}

\subsection{Discussions}

{Our results provide deterministic stability certificates for two representative LTV settings. The slowly varying case gives a time-step-wise PES bound without dwell time, but it requires the model drift to be small enough to preserve sequential stability. The piecewise-constant case allows larger jumps, but replaces global sequential stability by within-mode stability plus a dwell-time contraction across switches. These certificates are therefore complementary: one is suited to continuous drift, while the other is suited to abrupt but sufficiently separated changes. The constants are not meant to be tight design rules; rather, they make explicit how variation, excitation, process noise, and stepsize interact in the feedback loop.}


{The use of normalized sliding-window least squares is central to the analysis. Normalization removes the need for an a priori state bound in the identification error, and the sliding window makes the model estimation adaptive to local dynamics. The window length creates the usual bias--variance trade-off: a longer window can average noise, but it also mixes data from different frozen-time models and increases the variation-induced bias in Lemmas~\ref{lem:id_error} and~\ref{lem:id_error_sw}.  Exponentially weighted data with a forgetting factor could provide a smoother version of this trade-off, especially in stochastic settings, but it would introduce an additional tuning parameter and a different deterministic bias term. We leave the investigation of the forgetting factor mechanism to future work.}


{Another natural extension is to analyze direct data-driven policy optimization methods such as data-enabled policy optimization (DeePO) \cite{zhao2024data} for LTV systems. Different from the indirect approach in Algorithm \ref{algo2_switch}, DeePO updates the policy directly through a covariance parameterization and avoids explicit model estimation. In the LTV setting, the covariance and the frozen-time objective evolve online, and hence a stability proof would need to control the sample covariance and the policy update simultaneously, which is more delicate than the model-based perturbation argument used here.}

{Finally, one may replace the frozen-time infinite-horizon LQR objective by a discounted or finite-horizon surrogate, since future dynamics are unknown. Such objectives emphasize near-future behavior and may better match rapidly varying environments. However, discounted LQR policies are not automatically stabilizing for the undiscounted closed-loop system, and the choice of discount factor would interact with the {policy gradient stepsize} and the excitation signal~\cite{zhao2022sample}. Establishing stability certificates for these alternative objectives remains an interesting direction.}

\section{Numerical case studies} 
In this section, we illustrate the effectiveness of the proposed PGAC algorithm on LTV systems and nonlinear systems. {In particular, the nonlinear example validates the scenario when the LTV model used for adaptation is only a local approximation of the closed-loop dynamics.}    
\subsection{Continuously time-varying systems} 
We first consider a slowly time-varying system in the form of \eqref{equ:ltv_sys} adapted from \cite{dean2020sample} with
$$
	A_t = \underbrace{\begin{bmatrix}
		1.01 & 0.01 & 0 \\
		0.01 & 1.01 & 0.01 \\
		0 & 0.01 & 1.01
	\end{bmatrix}}_{A} + \alpha \sin\!\left(\frac{2\pi t}{T_v}\right)\Delta, ~B_t = I_3,
$$
where $\Delta = \mathrm{diag}(1,0.6,0.3)$ specifies the variation direction, 
$\alpha = 0.3$ controls the variation amplitude, and $T_v=200$ denotes the variation period. The penalty matrices are set to $Q = I_3, R = 10^{-3} I_3$.

The controller is updated using Algorithm \ref{algo2_switch} {with the normalized least-squares estimate in \eqref{equ:sysid}}. 
The sliding window length is set to $L=20$, and the policy gradient stepsize is $\eta = 0.05$. 
{The initial stabilizing policy is computed as the certainty-equivalence LQR gain for the initial frozen-time system.}  
During online operation, the control input is given by
$
u_t = K_t x_t + e_t,
$
where {$e_t$ is sampled uniformly from $[-10^{-2},10^{-2}]^3$ and the process noise is sampled uniformly from $[-2\times 10^{-3},2\times 10^{-3}]^3$}.

Fig.~\ref{fig:slow_var} illustrates the state norm $\|x_t\|$ and the spectral radius of the open-loop matrix $\rho(A_t)$.  
As indicated by Theorem \ref{thm:indirect}, the PGAC algorithm successfully stabilizes the system, and the state rapidly converges to a small neighborhood of the origin induced by the probing signal. {This happens despite the fact that the open-loop matrix $A_t$ is time-varying and unstable for parts of the trajectory, as reflected by $\rho(A_t)>1$ in Fig.~\ref{fig:slow_var}.}
{Fig.~\ref{fig:slow_gap} further reports the frozen-time optimality gap $({J_t-J_t^\star})/{J_t^\star}$, where $J_t$ is the LQR cost of the learned policy on the current frozen-time model and $J_t^\star$ is the corresponding optimal LQR cost. We also plot the gap of the fixed nominal LQR controller designed at the initial frozen-time model. In this slowly varying example, both gaps remain small, which is consistent with the fact that the nominal model is only mildly mismatched.}


\begin{figure}[t]
	\centering
	\includegraphics[width=0.9\linewidth]{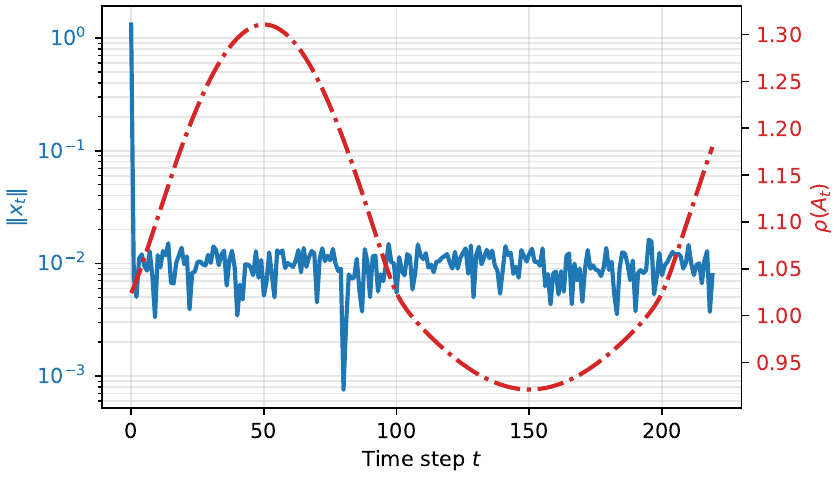}
	\caption{{Spectral radius $\rho(A_t)$ of the slowly time-varying system and the state norm $\|x_t\|$ under the PGAC controller with bounded uniform noise.}}
	\label{fig:slow_var}
\end{figure}

\begin{figure}[t]
	\centering
	\includegraphics[width=0.9\linewidth]{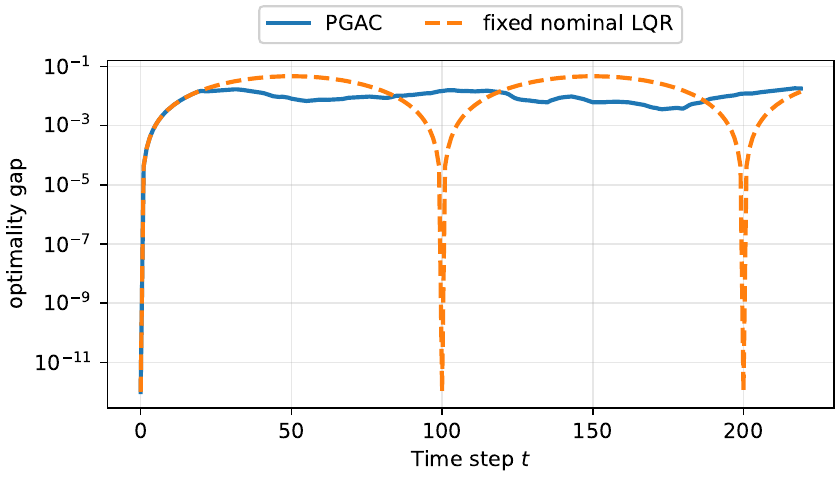}
	\caption{{Frozen-time optimality gap of PGAC and the fixed nominal LQR controller for the slowly time-varying system with bounded uniform noise.}}
	\label{fig:slow_gap}
\end{figure}

\subsection{Piecewise-constant time-varying systems}

Next, we consider a piecewise-constant system, where the system matrix $A_t$ switches among $\{A_1,A_2,A_3\}$, where
\[
{\begin{aligned}
A_1 &= A + 0.5\,\mathrm{diag}(1,0.6,0.3),\\
A_2 &= A - 0.5\,\mathrm{diag}(1,0.5,0.2),
\end{aligned}}
\]
\[
A_3 = A + 
\begin{bmatrix}
	0 & 0.010 & 0 \\
	0 & 0 & 0.008 \\
	0 & 0 & 0
\end{bmatrix}.
\]
The input matrix remains constant, i.e., $B_t = I_3$. The switching signal follows a periodic pattern with dwell time $\tau=20$. {The controller again uses normalized least-squares with $L=20$, while the probing and process noises are uniformly bounded as in the continuously time-varying example.}

Fig.~\ref{fig:piecewise_ltv} shows the spectral radius $\rho(A_t)$ together with the state norm $\|x_t\|$. 
The plot of $\rho(A_t)$ clearly illustrates the piecewise-constant switching behavior of the system dynamics. 
Despite these abrupt changes, the PGAC algorithm maintains closed-loop stability and keeps the state bounded. Fig.~\ref{fig:gap_piecewise} further shows the frozen-time optimality gap
$
({J_t - J_t^\star})/{J_t^\star},
$
where $J_t^\star$ denotes the optimal LQR cost corresponding to the current system $(A_t,B_t)$. 
{The fixed nominal LQR controller is included for comparison. The vertical jumps in the optimality gap occur at switching instants, after which PGAC readapts to the new mode.} The results demonstrate that the proposed algorithm is able to adapt to switching system dynamics and achieve near-optimal performance for each operating mode {despite bounded process noise}.

\begin{figure}[t]
	\centering
	\includegraphics[width=0.9\linewidth]{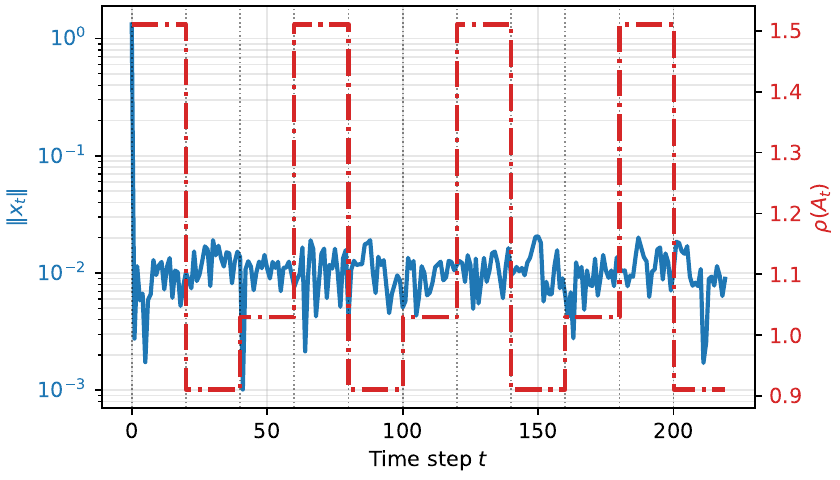}
	\caption{{Spectral radius $\rho(A_t)$ and state norm $\|x_t\|$ for the piecewise-constant time-varying system under PGAC with bounded uniform noise. Vertical dashed lines indicate switching instants.}}
	\label{fig:piecewise_ltv}
\end{figure}

\begin{figure}[t]
	\centering
	\includegraphics[width=0.9\linewidth]{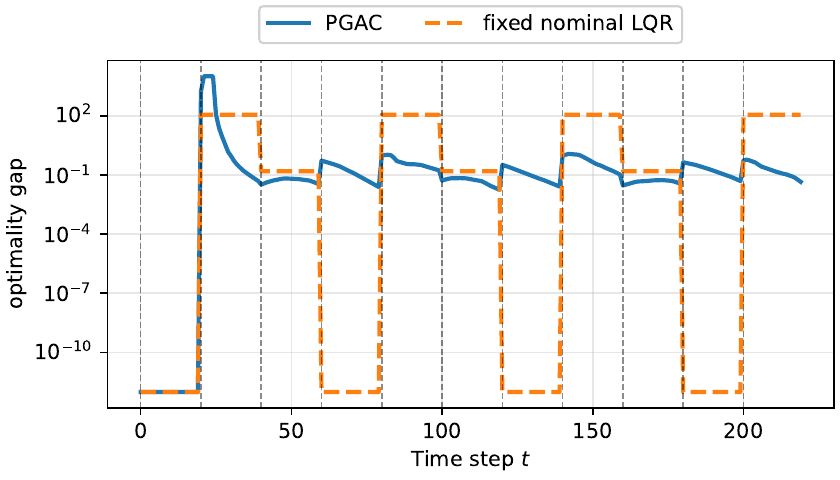}
	\caption{{Frozen-time optimality gap of PGAC and the fixed nominal LQR controller under switching dynamics and bounded uniform noise. Vertical dashed lines indicate switching instants.}}
	\label{fig:gap_piecewise}
\end{figure}

\subsection{{Nonlinear planar quadrotor}}

{Finally, we evaluate the proposed method on a nonlinear planar quadrotor model \cite{bartos2025stability}.} {The state is given by
$x_t=[p_{x,t},p_{z,t},v_{x,t},v_{z,t},\phi_t,\omega_t]^\top$, and the input is the hover-normalized thrust and torque $u_t=[u_{1,t},u_{2,t}]^\top$. The nonlinear dynamics are discretized by the Euler method:
\begin{align*}
	p_{x,t+1} &= p_{x,t}+h v_{x,t}+w_{1,t},\\
	p_{z,t+1} &= p_{z,t}+h v_{z,t}+w_{2,t},\\
	v_{x,t+1} &= v_{x,t}+h\left[-(g+u_{1,t})\sin(\phi_t)+d_t\right]+w_{3,t},\\
	v_{z,t+1} &= v_{z,t}+h\left[(g+u_{1,t})\cos(\phi_t)-g\right]+w_{4,t},\\
	\phi_{t+1} &= \phi_t+h\omega_t+w_{5,t},\\
	\omega_{t+1} &= \omega_t+h u_{2,t}/J_t+w_{6,t}.
\end{align*}
Here, $d_t$ models a horizontal wind disturbance and $J_t$ is the time-varying moment of inertia. We consider a persistent square-wave disturbance,
\begin{align*}
	d_t&=1.8\,\operatorname{sign}\!\left(\sin(2\pi t/260)\right),\\
	J_t&=1+0.60\,\operatorname{sign}\!\left(\sin(2\pi t/360)\right).
\end{align*}
The process noise and the probing signal are independently sampled from uniform distributions with bounded support, which is consistent with the bounded-noise setting considered in the theory.}

{We compare PGAC with a fixed nominal LQR controller designed from the hover linearization with nominal inertia and then kept constant throughout the run. The adaptive controller uses a sliding-window normalized least-squares estimate of the local linear dynamics, followed by the {policy gradient update} computed from the estimated model. Fig.~\ref{fig:quadrotor_square} reports the state norm, and Fig.~\ref{fig:quadrotor_disturbance} shows the disturbance profiles. The drift that starts around the middle of the trajectory coincides with a change in the square-wave wind profile; the fixed nominal controller cannot compensate for this persistent mismatch and drifts away from the operating point. In contrast, PGAC adapts from the recent closed-loop data and returns the state to a small neighborhood of the origin.}

\begin{figure}[t]
	\centering
	\includegraphics[width=0.9\linewidth]{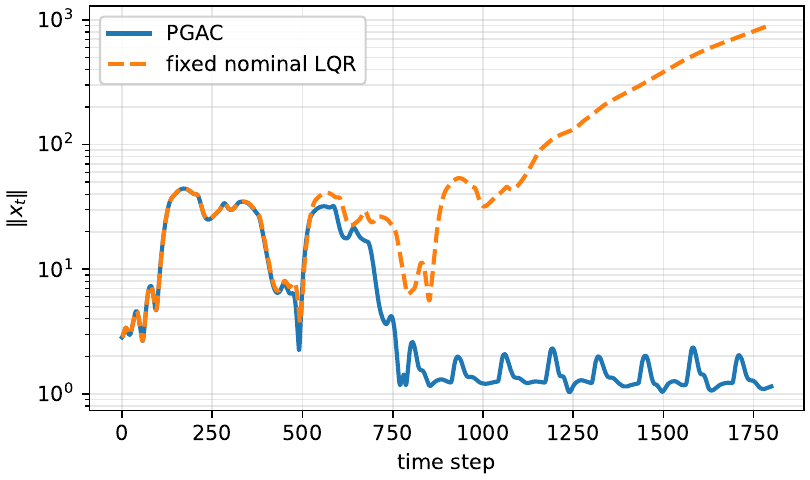}
	\caption{{State norm of the nonlinear planar quadrotor under persistent square-wave disturbance and bounded uniform process noise. PGAC adapts from closed-loop data using normalized least-squares, while the baseline is a fixed nominal LQR controller.}}
	\label{fig:quadrotor_square}
\end{figure}

\begin{figure}[t]
	\centering
	\includegraphics[width=0.9\linewidth]{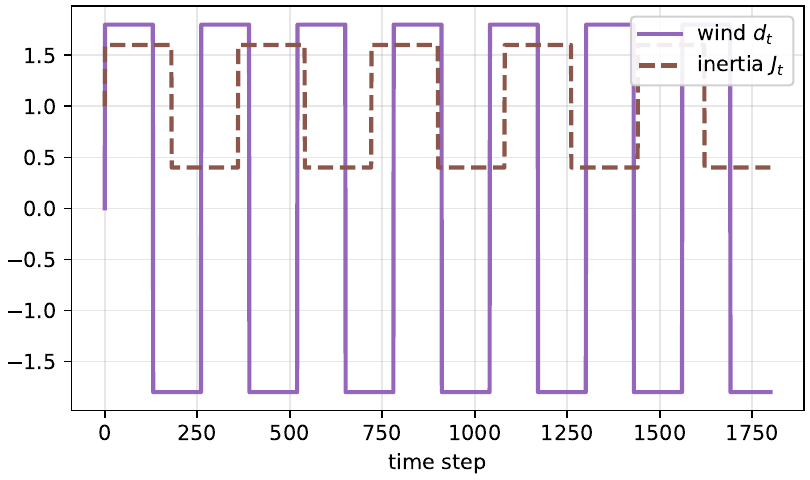}
	\caption{{Wind and inertia profiles used in the nonlinear planar quadrotor simulation.}}
	\label{fig:quadrotor_disturbance}
\end{figure}

\section{Conclusions}
This paper has proposed a policy gradient adaptive control method for stabilizing linear time-varying systems. We have shown practical stability for both slowly time-varying systems and piecewise-constant systems. Simulations have validated our theoretical results.

Future work includes extending the analysis to {high-probability} stochastic settings with process noise, where the identification accuracy and stability conditions depend on the signal-to-noise ratio. It is also of interest to study the optimal choice of the window length or alternative forgetting mechanisms in noisy environments. Finally, investigating discounted LQR formulations and their stability in adaptive control remains an interesting direction.

\appendix
\section{Proof of Lemma \ref{lem:id_error}}
The data matrices $(X_t,U_t,X_{t+1})$ satisfy the following dynamics
\begin{align*}
	&x_{t-i+1}
	= A_{t-i}x_{t-i}+B_{t-i}u_{t-i}+w_{t-i} \\
	&= A_tx_{t-i}+B_tu_{t-i}
	+(A_{t-i}-A_t)x_{t-i} \\
	&+(B_{t-i}-B_t)u_{t-i}+w_{t-i} \\
	&= A_tx_{t-i}+B_tu_{t-i}
	+\left(\sum_{j=1}^{i}(A_{t-j}-A_{t-j+1})\right)x_{t-i} \\
	&\quad
	+\left(\sum_{j=1}^{i}(B_{t-j}-B_{t-j+1})\right)u_{t-i}
	+w_{t-i} \\
	&\overset{\eqref{equ:variation}}{=}
	A_tx_{t-i}+B_tu_{t-i}
	+\sum_{j=1}^{i}\Delta_{t-j+1}
	[u_{t-i}^{\top},x_{t-i}^{\top}]^{\top}
	+w_{t-i} \\
	&:= A_tx_{t-i}+B_tu_{t-i}+v_{t-i}+w_{t-i},
	\qquad i\in\{1,\dots,L\}.
\end{align*}
Here, $v_{t-i}$ denotes the residual induced by the time variation of the
system matrices. Let
\[
V_t:=\begin{bmatrix}v_{t-L}&v_{t-L+1}&\cdots&v_{t-1}\end{bmatrix}.
\]
Then, $X_{t+1}=[B_t,A_t]D_t+V_t+W_t$. Multiplying both sides by $N_t^{-1}$
gives
\[
\bar X_{t+1}=[B_t,A_t]\bar D_t+\bar V_t+\bar W_t,
\]
where $\bar V_t:=V_tN_t^{-1}$ and $\bar W_t:=W_tN_t^{-1}$.

Since $\bar D_t$ has full row rank, the normalized least-squares estimate
satisfies $[\hat B_t,\hat A_t]=\bar X_{t+1}\bar D_t^\dagger$. Thus,
\begin{equation}\label{equ:bb}
	\begin{aligned}
		\|[\hat B_t,\hat A_t]-[B_t,A_t] \|
		&=
		\|(\bar V_t+\bar W_t)\bar D_t^\dagger\|\\
		&\leq \|\bar V_t\bar D_t^\dagger\|
		+
		\|\bar W_t\bar D_t^\dagger\|. 
	\end{aligned}
\end{equation} 
For the first term, since $\|\bar d_s\|\leq 1$, and
\[
\bar v_{t-i}:=v_{t-i}/n_{t-i}
=
\left(\sum_{j=1}^{i}\Delta_{t-j+1}\right)\bar d_{t-i},
\]
it follows from Assumption \ref{assum:variation} that $\|\bar v_{t-i}\|\leq i\delta$.
Furthermore, it holds that
\[
\|\bar V_t\|_F^2
=
\sum_{i=1}^{L}\|\bar v_{t-i}\|^2
\leq
\delta^2\sum_{i=1}^{L}i^2
\leq
\delta^2L^3,
\]
and hence $\|\bar V_t\|\leq \delta L^{3/2}$. 	Moreover, since Assumption \ref{assum:pe} implies
$\|\bar D_t^\dagger\|\leq 1/(\sqrt L\gamma)$, we have
\[
\|\bar V_t\bar D_t^\dagger\|
\leq
\|\bar V_t\|\|\bar D_t^\dagger\|
\leq
\frac{L\delta}{\gamma}.
\]
For the second term in \eqref{equ:bb}, since $n_t\geq 1$ and $\|w_t\|\leq w_m$,
\[
\|\bar W_t\|_F^2
=
\sum_{i=1}^{L}\|w_{t-i}/n_{t-i}\|^2
\leq
Lw_m^2.
\]
Hence $\|\bar W_t\|\leq \sqrt L w_m$, and
\[
\|\bar W_t\bar D_t^\dagger\|
\leq
\|\bar W_t\|\|\bar D_t^\dagger\|
\leq
\frac{w_m}{\gamma}.
\]
Combining the two bounds  completes the proof.

{\section{Proof of Lemma \ref{lem:id_error_sw}}
Fix a time $t$ with $T_i\leq t<T_{i+1}$. For the $\ell$-th column in the sliding window, write $(A_{t-\ell},B_{t-\ell})$ for the system matrices active at time $t-\ell$. Then
\[
\begin{aligned}
x_{t-\ell+1}
&=A_{t-\ell}x_{t-\ell}+B_{t-\ell}u_{t-\ell}
+w_{t-\ell},\\
&\hspace{3.6cm}\ell\in\{1,\ldots,L\}.
\end{aligned}
\]
Adding and subtracting the current mode $(A_i,B_i)$ gives
\[
x_{t-\ell+1}
=A_ix_{t-\ell}+B_iu_{t-\ell}+q_{t-\ell}+w_{t-\ell},
\]
where
\[
q_{t-\ell}:=([B_{t-\ell},A_{t-\ell}]-[B_i,A_i])d_{t-\ell}.
\]
Since at most $N=\lceil L/\tau\rceil$ switches can occur in a window of length $L$, the accumulated jump between any sample in the window and the current mode is bounded by
\[
\|[B_{t-\ell},A_{t-\ell}]-[B_i,A_i]\|\leq N\delta,
\qquad \ell\in\{1,\ldots,L\}.
\]
With $E_t:=\begin{bmatrix}q_{t-L}&q_{t-L+1}&\cdots&q_{t-1}\end{bmatrix}$, we therefore have
\[
X_{t+1}=[B_i,A_i]D_t+E_t+W_t .
\]
Multiplying by $N_t^{-1}$ gives
\[
\bar X_{t+1}=[B_i,A_i]\bar D_t+\bar E_t+\bar W_t,
\]
where $\bar E_t:=E_tN_t^{-1}$ and $\bar W_t:=W_tN_t^{-1}$. The normalized least-squares estimate satisfies
\[
[\hat B_t,\hat A_t]=\bar X_{t+1}\bar D_t^\dagger,
\]
and hence
\[
\|[\hat B_t,\hat A_t]-[B_i,A_i]\|
\leq
\|\bar E_t\bar D_t^\dagger\|
+
\|\bar W_t\bar D_t^\dagger\|.
\]
For the switching term, since $\bar d_{t-\ell}=d_{t-\ell}/n_{t-\ell}$ and $\|\bar d_{t-\ell}\|\leq 1$,
\[
\|q_{t-\ell}/n_{t-\ell}\|
\leq
\|[B_{t-\ell},A_{t-\ell}]-[B_i,A_i]\|\,\|\bar d_{t-\ell}\|
\leq N\delta .
\]
Thus $\|\bar E_t\|\leq \|\bar E_t\|_F\leq \sqrt L\,N\delta$. Assumption \ref{assum:pe} gives $\|\bar D_t^\dagger\|\leq 1/(\sqrt L\gamma)$, so
\[
\|\bar E_t\bar D_t^\dagger\|
\leq
\frac{N\delta}{\gamma}.
\]
Similarly, using $n_{t-\ell}\geq 1$ and $\|w_{t-\ell}\|\leq w_m$, we obtain $\|\bar W_t\|\leq \sqrt L w_m$ and therefore
\[
\|\bar W_t\bar D_t^\dagger\|
\leq
\frac{w_m}{\gamma}.
\]
Combining the two estimates proves
\[
\|[\hat B_t,\hat A_t]-[B_i,A_i]\|
\leq
\frac{N\delta+w_m}{\gamma},
\]
which completes the proof.}

\section{Proof of Theorem \ref{thm:indirect}}	 
We first provide some useful lemmas from \cite{zhao2025policy}.
\begin{lemma}\label{lem:bounds}
	For $K \in \mathcal{S}$, it holds
	(\romannumeral1) 
	$
	\|\Sigma\|  \leq {C(K)}/{\underline{\sigma}(Q)},
	$
	(\romannumeral2) 
	$
	\|P\| \leq C(K),
	$
	and (\romannumeral3) 
	$
	\|K\|_F \leq ({C(K)}/{\underline{\sigma}(R)})^{{1}/{2}}.
	$
\end{lemma}
   
\begin{lemma}[{Lyapunov perturbation}]\label{lem:perturb}
	Let $A\in \mathbb{R}^{n \times n}$ be stable and $\Sigma(A)$ be the unique positive definite solution to $\Sigma(A) = I_n + A\Sigma(A) A^{\top}$. If 
	$
	\|A'-A\| \leq 1/({4\|\Sigma(A)\|(1+a_m)}),
	$
	then $A'$ is stable and
	$
	\|\Sigma(A')-\Sigma(A)\| \leq 4 \|\Sigma(A)\|^2(1+a_m)\|A'-A\|.
	$
\end{lemma}

At time $t$, consider the policy gradient update with estimated $(\hat{A}_t, \hat{B}_t)$
\begin{equation}\label{equ:pg_appr}
	K' = K - \eta \nabla \hat{C}_t(K),
\end{equation}
and the update with the ground-truth $(A_t,B_t)$
$$
	K'' = K - \eta \nabla {C}_t(K).
$$
It follows from \cite{fazel2018global} that, if $\eta \leq 1/l_t$, then
\begin{equation}\label{equ:exact}
	C_t(K'') - C_t(K) \leq  -\frac{\eta}{2\mu_t} (C_t(K) - C_t^*), 
\end{equation} 
where $l_t$ is the smoothness parameter and $\mu_t$ is the gradient dominance constant corresponding to $(A_t, B_t)$.
To show the convergence of \eqref{equ:pg_appr}, we first quantify the distance between the exact gradient $\nabla {C}_t(K)$ and the {approximate} gradient {$\nabla\hat{C}_t(K)$}. Let $p_1$ be a scalar function 
$$
p_1 = \frac{8C_t^3(K)}{\underline{\sigma}^2(Q)}\left(1+\frac{C_t(K)}{\underline{\sigma}(Q)}\right) \left(1+\sqrt{\frac{C_t(K)}{\underline{\sigma}(R)}}\right),
$$
and $p_2 = {C_t^2(K) }/({\underline{\sigma}(Q)p_1})$. Then, we have the following results, which can be proved using Lemma \ref{lem:perturb} and following similar derivation of \cite[Lemmas 12-14]{zhao2025policy}. Define $\mathcal{S}_t:=\{K\in \mathbb{R}^{m\times n}|\rho(A_t+B_tK)<1\}$.
\begin{lemma}\label{lem:grad_diff}
	Let $K \in \mathcal{S}_t$. Then, there exists a polynomial  
	$p_3 = \text{poly}(C_t(K)/\underline{\sigma}(Q), a_m, b_m, \|R\|, 1/\underline{\sigma}(R))$ such that, if {$\epsilon \leq p_2$}, then $
	\| \nabla{C}_t(K) - \nabla\hat{C}_t(K) \| \leq  {p_3 } {\epsilon}.$
\end{lemma} 
\begin{lemma}\label{lem:cost_diff}
	Let $K \in \mathcal{S}_t$. There exist polynomials $p_4 = \text{poly}(C_t(K)/\underline{\sigma}(Q), a_m^{-1}, b_m^{-1}, \|R\|^{-1}, \underline{\sigma}(R) )$ and $p_5,p_6 = \text{poly}(C_t(K)/\underline{\sigma}(Q), a_m, b_m, \|R\|, 1/\underline{\sigma}(R))$ such that, if $\|\widetilde{K} - K\| \leq p_4$, then  $\widetilde{K} \in \mathcal{S}_t$ and
	$$
	\|\widetilde{\Sigma} - \Sigma\| \leq p_5\|\widetilde{K} - K\|, ~~
	|C_t(\widetilde{K}) - C_t(K)| \leq p_6\|\widetilde{K} - K\|.
	$$ 
\end{lemma}
\begin{lemma}\label{lem:bias}
	Let $K \in \mathcal{S}_t$. There exists a polynomial $p_7$ in $({\underline{\sigma}(Q)}/{C_t(K)}, a_m^{-1},  b_m^{-1},  \|R\|^{-1}, \underline{\sigma}(R))$ such that, if
	$$
	{\epsilon} \leq  p_2 ~~\text{and}~~ \eta \leq \min\left\{\frac{p_4 }{p_3 {\epsilon}}, p_7\right\},
	$$
	then it holds that
	$
	\left| C_t(K'') - C_t(K')\right| \leq  \eta p_3p_6 {\epsilon}.
	$
\end{lemma}

With Lemma \ref{lem:bias} and \eqref{equ:exact}, we show the convergence of \eqref{equ:pg_appr}
\begin{equation}\label{equ:appro}
	C_t(K') - C_t(K) \leq  -\frac{\eta}{2\mu_t} (C_t(K) - C_t^*) + \eta p_3p_6 {\epsilon}.  
\end{equation}

To show the convergence of Algorithm \ref{algo2_switch}, we need to first prove that $C_t(K_t)$ is uniformly upper-bounded, such that the polynomials $p_i, i\in\{1,2,\dots, 7\}$ have {uniform} bounds. Define $\underline{C^*}:=\min_{t\geq t_0} {C_t^*}$ and $\overline{C^*}:=\max_{t\geq t_0} {C_t^*}$, which exists as $(A_t, B_t)$ is uniformly bounded according to Assumption \ref{ass:bounded}. Let $\underline{l} := l(\underline{C^*})$, $\underline{\mu} := {\underline{C^*}}/{(\underline{\sigma}(R)\|Q\|)}$, and 
\begin{equation}\label{equ:bound_cost}
	\overline{C} := \overline{C^*} + 2 + \frac{1}{\underline{l}\underline{\mu}} + C_{t_0}(K_{t_0}).
\end{equation}

\begin{lemma}\label{lem:cost_diff_model}
	Let $K \in \mathcal{S}_t$. There exist polynomials $p_9 = \text{poly}(1/(C_t(K))^{1/2}, a_m^{-1}, b_m^{-1}, \|R\|^{-1}, \underline{\sigma}(R), \underline{\sigma}(Q) )$ and $p_{10} = \text{poly}(C_t(K)/\underline{\sigma}(Q), a_m, b_m, \|R\|, 1/\underline{\sigma}(R))$ such that, if $\delta \leq p_9$, then  $K \in \mathcal{S}_{t+1}$ and
	$$ 
	|C_{t+1}(K) - C_t(K)| \leq p_{10}\delta.
	$$ 
\end{lemma}
\begin{proof}
The closed-loop matrix variation is $\|A_{t+1}+B_{t+1}K - (A_{t}+B_{t}K) \| \leq \delta(1+\|K\|_F)$. By Lemma \ref{lem:perturb}, if
$$
\delta \leq \frac{1}{4(1 + \|K\|_F)\|\Sigma_t\|(1+\|\Sigma_t\|)},
$$
then 
$
\|\Sigma_{t+1} - \Sigma_t\|  \leq  4\delta \|\Sigma_t\|^2(1+\|A_t + B_tK\|)(1+\|K\|_F)
$.
Noting
$
	|C_{t+1}(K) - C_t(K)|  \leq \text{Tr}(Q+K^{\top}RK)\|\Sigma_{t+1} - \Sigma_t\|  
$
and that both $\Sigma_t$ and $\|K\|_F$ are upper bounded by $C_t(K)$, the proof {is complete}.
\end{proof}

By Assumption \ref{ass:bounded}, {the gradient-dominance constants are uniformly upper bounded}, i.e., $\mu_t \leq \bar{\mu}$ for some constant $\bar{\mu}$. Since the quantities $l$ and $p_i$ are {functions} of $C(K)$,  let $\bar{l}, \bar{p}_1, \underline{p}_2, \bar{p}_3, \bar{p}_5, \bar{p}_6, \underline{p}_7, \underline{p}_9, \bar{p}_{10}$ be the associated quantities at {$\overline{C}+1$}, and $\underline{p}_4$ be the quantity at $\underline{C^*}$.  Then, we have the following results.

\begin{lemma}[Boundedness of the cost]\label{lem:bounded}
	If
	\begin{equation}\label{equ:cond}
		\begin{aligned}
			\delta &\leq  \min\left\{\underline{p}_9, {\frac{1}{\bar p_{10}},} \frac{1}{\bar{p}_{10}(1+2\bar{l}\bar{\mu})}\right\}, \\
			{\epsilon} &\leq \min\left\{\underline{p}_2,\frac{1}{2\bar{p}_6\bar{p}_3\bar{\mu}}\right\},\\
			\eta &\leq \min\left\{\frac{\underline{p}_4}{\bar{p}_3{\epsilon}}, \underline{p}_7, \frac{1}{\bar{l}} \right\},
		\end{aligned} 
	\end{equation}
	then $C_t(K_t)$ of {Algorithm~\ref{algo2_switch}} has a uniform upper bound, i.e.,
	$C_t(K_t) \leq \overline{C}$ with $\overline{C}$ defined in \eqref{equ:bound_cost}.
\end{lemma} 
\begin{proof}
	The proof is based on mathematical induction. Clearly, the bound holds at $t = t_0$, i.e., $C_{t_0}(K_{t_0}) \leq \overline{C}$. Suppose that $C_t(K_t) \leq \overline{C}$ for {some $t\geq t_0$}. Next, we show $C_{t+1}(K_{t+1}) \leq \overline{C}$.

	{By Lemma~\ref{lem:cost_diff_model} and $\delta\leq \min\{\underline p_9,1/\bar p_{10}\}$, the induction hypothesis also gives $K_t\in\mathcal S_{t+1}$ and $C_{t+1}(K_t)\leq \overline C+1$. Thus, the uniform constants chosen above apply to the update at time $t+1$.} By Lemma \ref{lem:bias}, \eqref{equ:exact}, and our hypothesis $C_t(K_t) \leq \overline{C}$, the gradient descent yields
	\begin{align*}
		&C_{t+1}(K_{t+1}) - C_{t+1}(K_t) \\
		&\leq -\frac{\eta}{2\mu_{t+1}} (C_{t+1}(K_t) - C_{t+1}^*) + \eta \bar{p}_6\bar{p}_3 {\epsilon}\\
		&\leq -\frac{\eta}{2\bar{\mu}} (C_{t+1}(K_t) - C_{t+1}^*) + \frac{\eta}{2\bar{\mu}}
	\end{align*}	
	
		By Lemma \ref{lem:cost_diff_model} and our condition on $\delta \leq \underline{p}_9$, the cost function $C_{t+1}(K_{t})$ associated with $(A_{t+1}, B_{t+1})$ can be upper bounded, i.e.,
	$
	|C_{t+1}(K_t) - C_t(K_t)| \leq \bar{p}_{10}\delta.
	$ Then, it follows that
	\begin{align*}
		&C_{t+1}(K_{t+1}) - C_{t}(K_t) \\
		&\leq -\frac{\eta}{2\bar{\mu}} (C_{t}(K_t) - C_{t+1}^*) + \frac{\eta}{2\bar{\mu}} +\frac{\eta\bar{p}_{10}\delta}{2\bar{\mu}} + \bar{p}_{10}\delta \\
		&\leq -\frac{\eta}{2\bar{\mu}} (C_{t}(K_t) - C_{t+1}^*) + \frac{\eta}{\bar{\mu}}\\
		&\leq -\frac{\eta}{2\bar{\mu}} (C_{t}(K_t) - \overline{C^*}) + \frac{\eta}{\bar{\mu}},
	\end{align*}
	where the last inequality follows from  $C_{t+1}^* \leq \overline{C^*}$.
	
	Consider two cases. If $C_t(K_t) \geq \overline{C^*} +2$, then 
	$$
	C_{t+1}(K_{t+1}) \leq C_t(K_t) - \frac{\eta}{{\bar{\mu}}}  + \frac{\eta}{{\bar{\mu}}} = C_t(K_t) \leq \overline{C}.
	$$
	Otherwise, if $C_t(K_t) < \overline{C^*} +2$, then
	$$
	{C_{t+1}}(K_{t+1}) \leq \overline{C^*} +2 + \frac{\eta}{\bar{\mu}} \leq \overline{C^*} + 2 + \frac{1}{\bar{l}\bar{\mu}}\leq \overline{C}.
	$$
	The proof {is complete}.
\end{proof}

	Next, we show the sequential stability of the closed-loop system {under Algorithm~\ref{algo2_switch}}.

\begin{lemma}\label{lem:seq_stable}
	There exist $\bar{p}_{8}, \bar{p}_{11}, \bar{p}_{12}$ as functions of $\overline{C}$ such that, if
	\begin{equation}\label{equ:condition}
		\begin{aligned}
			\delta &\leq  \min\left\{\underline{p}_9, {\frac{1}{\bar p_{10}},} \frac{1}{\bar{p}_{10}(1+2\bar{l}\bar{\mu})}, \bar{p}_{11}\right\}, \\
			{\epsilon} &\leq \min\left\{\underline{p}_2,\frac{1}{2\bar{p}_6\bar{p}_3\bar{\mu}}\right\},\\
			\eta &\leq \min\left\{\frac{\underline{p}_4}{\bar{p}_3{\epsilon}}, \underline{p}_7, \frac{1}{\bar{l}}, \bar{p}_{12} \right\},
		\end{aligned} 
	\end{equation} 
	then  $\{K_t\}$ of {Algorithm~\ref{algo2_switch}} is {sequentially strongly stable} with parameters $(\kappa, \alpha)$, where 
	\begin{equation}\label{equ:expka}
		\kappa = \sqrt{\frac{\overline{C}}{\min\{\underline{\sigma}(R), \underline{\sigma}(Q)\}}},~\alpha = 1 - \sqrt{1-\frac{1}{\kappa^2}}.
	\end{equation} 
\end{lemma}

\begin{proof}
	By Lemma \ref{lem:bounded}, the cost is uniformly bounded, i.e., {$C_t(K_t)\leq \overline{C}$} for all $t\geq t_0$. Then, with the parameters $\kappa, \alpha$, the first two conditions (i) and (ii) in Definition \ref{def:sss} are satisfied, i.e., the policy $K_t$ is $(\kappa, \alpha)$-strongly stable. Further, it suffices to show (iii) $\|H_{t+1}^{-1}H_t\| \leq 1 + \alpha/2$, or equivalently,
	$\|\Sigma_{t+1}^{-1}\Sigma_t\| \leq (1 + {\alpha}/{2})^2,$ where
	\begin{align*}
		\Sigma_{t+1}
		&= I_n + (A_{t+1} + B_{t+1}K_{t+1})\Sigma_{t+1}\\
		&\qquad\cdot(A_{t+1} + B_{t+1}K_{t+1})^{\top}, \\
		\Sigma_{t}
		&= I_n + (A_{t} + B_{t}K_{t})\Sigma_{t}
		(A_{t} + B_{t}K_{t})^{\top}.
	\end{align*}
	{Indeed, the Lyapunov equation gives a positive definite covariance $\Sigma_t$ for each frozen-time closed-loop matrix. Choosing the similarity factor $H_t=\Sigma_t^{1/2}$, condition (iii) follows from
	$\|H_{t+1}^{-1}H_t\|^2
	=\|\Sigma_{t+1}^{-1/2}\Sigma_t\Sigma_{t+1}^{-1/2}\|
	=\|\Sigma_{t+1}^{-1}\Sigma_t\|$,}
	{where the last equality uses similarity of the two positive definite products.}
	
	By the perturbation theory for matrix inverse \cite[Theorem 35]{fazel2018global}, if $\|\Sigma_{t+1} - \Sigma_t\|<1/2$, then
	$
	\|\Sigma_{t+1}^{-1} - \Sigma_t^{-1}\| \leq {2\|\Sigma_{t+1} - \Sigma_t\|}/{\underline{\sigma}(\Sigma_t)} \leq 2\|\Sigma_{t+1} - \Sigma_t\|.
	$ Further,
	\begin{align*}
		\|\Sigma_{t+1}^{-1}\Sigma_t\|  
		& = \|(\Sigma_{t+1}^{-1} - \Sigma_t^{-1})\Sigma_t + I_n \|\\
		&\leq 1 + 2\|\Sigma_{t+1} - \Sigma_t\|\|\Sigma_t\| \\
		&\leq 1 + 2 \|\Sigma_{t+1} - \Sigma_t\|  {{C_t}(K_t)}/{\underline{\sigma}(Q)}  \\
		&\leq 1 + 2\kappa^2 \|\Sigma_{t+1} - \Sigma_t\|.
	\end{align*}
	
	Thus, we require $ \|\Sigma_{t+1} - \Sigma_t\|$ to be sufficiently small. By {the same Lyapunov perturbation argument as in Lemma~\ref{lem:cost_diff}}, if $\|A_{t+1}+B_{t+1}K_{t+1} - (A_t+B_tK_t)\| \leq \underline{p}_4$, then $\|\Sigma_{t+1} - \Sigma_t\| \leq \bar{p}_5 \|A_{t+1}+B_{t+1}K_{t+1} - (A_t+B_tK_t)\|$. Since
	 \begin{align*}
	 	& [B_{t+1}, A_{t+1}]\begin{bmatrix}
	 		K_{t+1} \\ I_n
	 	\end{bmatrix} - [B_t, A_t]\begin{bmatrix}
	 		K_t \\ I_n
	 	\end{bmatrix} \\
	 	&= ([B_{t+1}, A_{t+1}] - [B_t, A_t])\begin{bmatrix}
	 		K_t \\ I_n
	 	\end{bmatrix} \\
	 	&{+ [B_{t+1}, A_{t+1}]\left(\begin{bmatrix}
	 		K_{t+1} \\ I_n
	 	\end{bmatrix} - \begin{bmatrix}
	 		K_t \\ I_n
	 	\end{bmatrix}\right),}
	 \end{align*}
	 it holds that
	 \begin{align*}
	 	&\|A_{t+1}+B_{t+1}K_{t+1} - (A_t+B_tK_t)\| \\
	 	&\leq \delta(1+{\|K_t\|_F}) + \|K_{t+1} - K_t\|(a_m+b_m)\\
	 	&=\delta(1+{\|K_t\|_F}) + \eta\|\nabla \hat{C}_t(K_t)\|(a_m+b_m).
	 \end{align*}
	 We first provide a bound for $\|\nabla \hat{C}_t(K_t)\|$. By Lemma \ref{lem:grad_diff}, we have 
	 $
	 \| \nabla\hat{C}_t({K_t}) \| \leq \|\nabla{C}_t({K_t}) \| + {\bar{p}_3 }{\epsilon} \leq \|\nabla{C}_t({K_t}) \| + \bar{p}_3 \underline{p}_2.
	 $
	 Since {$C_t(K_t)\leq \overline{C}$}, the right-hand side has a uniform upper bound denoted by $\bar{p}_8$. Noting that {$\|K_t\|_F \leq (\overline{C}/{\underline{\sigma}(R)})^{{1}/{2}}$},
	 to ensure $\|A_{t+1}+B_{t+1}K_{t+1} - (A_t+B_tK_t)\| \leq \underline{p}_4$ and $\|\Sigma_{t+1} - \Sigma_t\|<1/2$ it suffices to let
	 \begin{equation}\label{equ:condi1}
	 		 \begin{aligned}
	 		&\delta(1+(\overline{C}/{\underline{\sigma}(R)})^{{1}/{2}}) \leq \min\left\{\frac{\underline{p}_4}{2}, \frac{1}{4\bar{p}_5} \right\}, \\
	 		&\eta\bar{p}_8(a_m+b_m) \leq \min\left\{\frac{\underline{p}_4}{2}, \frac{1}{4\bar{p}_5} \right\}.
	 	\end{aligned}
	 \end{equation}

	  Furthermore, it follows that $\|\Sigma_{t+1}^{-1}\Sigma_t\| \leq  1 + 2\kappa^2 \|\Sigma_{t+1} - \Sigma_t\| \leq 1 + 2\kappa^2(\delta(1+(\overline{C}/{\underline{\sigma}(R)})^{{1}/{2}}) + \eta\bar{p}_8(a_m+b_m))$. Since we also require $\|\Sigma_{t+1}^{-1}\Sigma_t\| \leq (1 + {\alpha}/{2})^2$, it suffices to let $2\kappa^2(\delta(1+(\overline{C}/{\underline{\sigma}(R)})^{{1}/{2}}) + \eta\bar{p}_8(a_m+b_m)) \leq \alpha$, i.e.,
	  \[
	  {\begin{aligned}
	  	 \delta\left(1+\left(\overline{C}/\underline{\sigma}(R)\right)^{{1}/{2}}\right)
	  	 &\leq \frac{\alpha}{4\kappa^2},\\
	  	 \eta\bar{p}_8(a_m+b_m)
	  	 &\leq \frac{\alpha}{4\kappa^2}.
	  \end{aligned}}
	  \]
	  Together with \eqref{equ:condi1}, we denote the bound on $\delta$ and $\eta$ as $\delta \leq \bar{p}_{11}$ and $\eta \leq \bar{p}_{12}$, respectively.  The proof {is complete}. 
\end{proof}

{It remains to show that sequential stability leads to the state bound \eqref{equ:bound_state}. Under Algorithm~\ref{algo2_switch}, the closed-loop dynamics can be written as}
\[
{x_{t+1}=(A_t+B_tK_t)x_t+B_te_t+w_t.}
\]
{Since $\{K_t\}$ is sequentially stable with parameters $(\kappa,\alpha)$, the standard input-to-state estimate for sequentially stable systems gives \cite{cohen2019learning}}
\[
{\|x_t\|\leq \kappa\left(1-\frac{\alpha}{2}\right)^{t-t_0}\|x_{t_0}\|
+\frac{2\kappa}{\alpha}\sup_{s\geq t_0}\|B_se_s+w_s\|.}
\]
{Using $\|B_s\|\leq b_m$, $\|e_s\|\leq e_m$, and $\|w_s\|\leq w_m$, we have $\sup_{s\geq t_0}\|B_se_s+w_s\|\leq b_me_m+w_m$. Renaming $(\kappa,\alpha)$ as $(\nu_5,\nu_6)$ yields \eqref{equ:bound_state}.}

{Finally, we prove the cumulative optimality-gap bound. Under the conditions of Lemma~\ref{lem:seq_stable}, the uniform constants in \eqref{equ:appro} give}
\begin{equation}\label{equ:regret_step_ltv}
{
C_t(K_{t+1})-C_t(K_t)
\leq
-\frac{\eta}{2\bar\mu}\big(C_t(K_t)-C_t^*\big)
+\eta \bar p_3\bar p_6\epsilon .}
\end{equation}
{Therefore, for every $T>t_0$, it holds that}

\begin{equation}\label{equ:sddd}
\begin{aligned}
	&\sum_{t=t_0}^{T-1}\big(C_t(K_t)-C_t^*\big)\\
	&\leq
	\frac{2\bar\mu}{\eta}
	\sum_{t=t_0}^{T-1}\big(C_t(K_t)-C_t(K_{t+1})\big)
	+2\bar\mu\bar p_3\bar p_6(T-t_0)\epsilon .
\end{aligned}
\end{equation}

{The first sum is   telescoping. Indeed,}
\[
{\begin{aligned}
&\sum_{t=t_0}^{T-1}\big(C_t(K_t)-C_t(K_{t+1})\big)\\
&=
C_{t_0}(K_{t_0})-C_{T-1}(K_T)\\
&\quad+
\sum_{t=t_0}^{T-2}\big(C_{t+1}(K_{t+1})
-C_t(K_{t+1})\big).
\end{aligned}}
\]
{Since the costs are nonnegative and Lemma~\ref{lem:cost_diff_model} applies uniformly on the bounded-cost set, it follows that}
\[
{\begin{aligned}
&\sum_{t=t_0}^{T-1}\big(C_t(K_t)-C_t(K_{t+1})\big)\\
&\qquad\leq
C_{t_0}(K_{t_0})+(T-t_0-1)\bar p_{10}\delta .
\end{aligned}}
\]
{Substituting this estimate into \eqref{equ:sddd} gives}
\[
{\begin{aligned}
&\sum_{t=t_0}^{T-1}\big(C_t(K_t)-C_t^*\big)\\
&\leq
\frac{2\bar\mu C_{t_0}(K_{t_0})}{\eta}
+\frac{2\bar\mu\bar p_{10}(T-t_0)\delta}{\eta}
+2\bar\mu\bar p_3\bar p_6(T-t_0)\epsilon.
\end{aligned}}
\]
{Dividing both sides by $T-t_0$ yields \eqref{equ:regret_ltv} after renaming constants.}

\section{Proof of Theorem \ref{thm:switch}}

Define
\begin{equation}\label{equ:upper}
\overline{C} = \overline{C^*} + 1 + \frac{1}{4\bar{l}\bar{\mu}} + {p}_{10}^*{p}_{9}^* + C_0(K_{t_0}),
\end{equation}
where ${p}_{9}^*, {p}_{10}^*$ are constants associated with $\overline{C^*} +1$.
We show that the cost function is uniformly upper bounded, i.e., $C_i(K_t)\leq \overline{C}$ for $t>t_0$, by using mathematical induction over switching intervals. At the beginning $t_0 = T_0$, it holds that $C_0(K_{t_0})\leq \overline{C}$ according to the definition \eqref{equ:upper}.

 
Consider that we are in the $i$-th mode, i.e., $T_i \leq t <T_{i+1}$. Assume that we have $C_{i-1}(K_{T_i}) \leq \overline{C^*} +1 < \overline{C}$. Then, after the system changes at $T_i$, the cost becomes $C_{i}(K_{T_i})$. By Lemma \ref{lem:cost_diff_model},  if $\delta \leq p_9^*$, then 
	$
	|C_{i}(K_{T_i}) - C_{i-1}(K_{T_i})| \leq p_{10}^*\delta.
	$
By the hypothesis $C_{i-1}(K_{T_i}) \leq \overline{C^*} +1$, the  cost  $C_{i}(K_{T_i})$ is upper bounded as $C_{i}(K_{T_i}) \leq \overline{C^*} +1 + {p}^*_{10}{p}^*_{9} <\overline{C}$ under the condition $\delta \leq p^*_9$.  Next, we show that under the given dwell time, the cost $C_{i}(K_{T_{i+1}})$ can be reduced to $ \overline{C^*} +1$.

Under the condition  
\begin{equation}\label{equ:condition27}
		{\begin{aligned}
		\delta &\leq  \min\left\{p_9^*,\frac{1}{\bar p_{10}^*}\right\},\\
		\epsilon_{\rm sw}
		&\leq \min\left\{\underline{p}_2,
		\frac{1}{4\bar{p}_6\bar{p}_3\bar{\mu}}\right\},\\
		\eta
		&\leq \min\left\{\frac{\underline{p}_4}{\bar{p}_3\epsilon_{\rm sw}},
		\underline{p}_7, \frac{1}{\bar{l}} \right\}.
		\end{aligned}}
\end{equation}
for $T_i \leq t <T_{i+1}$, the gradient descent satisfies
\begin{equation}
	C_{i}(K_{t+1}) - C_{i}(K_t) 
	\leq -\frac{\eta}{2\mu_{i}} (C_{i}(K_t) - C_{i}^*) + \eta \bar{p}_6\bar{p}_3 {\epsilon_{\rm sw}},
\end{equation}	 
Then,  we have
\begin{align*}
&C_{i}(K_{t}) - C_i^* \leq \left(1 - \frac{\eta}{2\mu_i}\right)^{t-T_i}(C_{i}(K_{T_i}) - C_i^*) \\
&+ \eta \bar{p}_6\bar{p}_3 {\epsilon_{\rm sw}} \sum_{j = T_i}^{t}\left(1-\frac{\eta}{2\mu_i}\right)^{t-j} \\
&\leq \left(1 - \frac{\eta}{2\mu_i}\right)^{t-T_i}(C_{i}(K_{T_i}) - C_i^*) + 2\mu_i \bar{p}_6\bar{p}_3 {\epsilon_{\rm sw}}   \\
&\leq \left(1 - \frac{\eta}{2\mu_i}\right)^{t-T_i}\overline{C} + \frac{1}{2},
\end{align*}
where the last inequality follows from \eqref{equ:condition27}.

Under the condition on the dwell time
$$
\tau \geq  \left\lceil  \frac{- \log(2\overline{C})}{\log(1-\frac{\eta}{2{\bar{\mu}}})} \right\rceil,
$$
the previous inequality further leads to $C_{i}(K_{T_{i+1}}) \leq  C_i^* + 1 \leq \overline{C^*} + 1$. 

Now, we prove $C_i(K_t) \leq \overline{C}$, for $T_i\leq t < T_{i+1}$. It holds 
\begin{align*}
	&C_{i}(K_{t+1}) - C_{i}(K_{t}) \\
	&\leq - \frac{\eta}{2\mu_i}(C_{i}(K_{t}) - C_i^*) 
	+ \eta \bar{p}_6\bar{p}_3 {\epsilon_{\rm sw}}  \\
	& \leq - \frac{\eta}{2\bar{\mu}}(C_{i}(K_{t}) - C_i^*) 
	+ \frac{\eta}{4\bar{\mu}} \\
	&  \leq - \frac{\eta}{2\bar{\mu}}(C_{i}(K_{t}) - \overline{C^*}) 
	+ \frac{\eta}{4\bar{\mu}}.
\end{align*} 
Consider two cases. If $C_i(K_t) \geq \overline{C^*} +1/2$, then 
$$
C_{i}(K_{t+1}) \leq C_i(K_t) - \frac{\eta}{4{\bar{\mu}}}  + \frac{\eta}{4{\bar{\mu}}} = C_i(K_t) \leq \overline{C}.
$$
Otherwise, if {$C_i(K_t)$} < $\overline{C^*} +1/2$, then
$$
C_i(K_{t+1}) \leq \overline{C^*} +\frac{1}{2} + \frac{\eta}{4\bar{\mu}} \leq \overline{C^*} +\frac{1}{2} + \frac{1}{4\bar{l}\bar{\mu}}\leq \overline{C}.
$$ 

So far, we have proved that $C_i(K_t) \leq \overline{C}$ for $T_i\leq t \leq T_{i+1}$. Furthermore, we can show that the closed-loop system is sequentially stable in the interval $T_i \leq t < T_{i+1}$.

\begin{lemma} 
	There exist $\bar{p}_{12}$ as a function of $\overline{C}$ such that, if
	\[
		{\begin{aligned}
		\delta &\leq  \min\left\{p_9^*,\frac{1}{\bar p_{10}^*}\right\},\\
		\epsilon_{\rm sw}
		&\leq \min\left\{\underline{p}_2,
		\frac{1}{4\bar{p}_6\bar{p}_3\bar{\mu}}\right\},\\
		\eta &\leq \min\left\{\frac{\underline{p}_4}{\bar{p}_3\epsilon_{\rm sw}},
		\underline{p}_7, \frac{1}{\bar{l}}, \underline{p}_{12} \right\},
		\end{aligned}}
	\]
	then $\{K_t\}, T_i\leq t \leq T_{i+1}$ is sequentially   stable for the system $(A_i, B_i)$ with parameters $(\kappa, \alpha)$, where $\kappa, \alpha$ are given by \eqref{equ:expka}.
\end{lemma}

The proof is similar to that of Lemma \ref{lem:seq_stable} and omitted.

The difficulty arises when switching occurs: the closed-loop matrix may change significantly and fail to satisfy sequential stability, which can cause the state norm to increase abruptly. To address this, we show that (i) the state remains upper bounded under a dwell-time condition, and (ii) the state norm eventually decreases to a constant level after the switching phase. As the sequence $\{K_t\}$ for $T_i\leq t < T_{i+1}$ is \((\kappa, \alpha)\)-sequentially stable, it holds that
\begin{equation}\label{22}
{\begin{aligned}
\|x_t\| \leq{}& \kappa \left(1 - \frac{\alpha}{2}\right)^{t-T_{i}} \|x_{T_{i}}\|\\
&+  \frac{2  \kappa}{\alpha}(b_me_m+w_m),
\qquad T_i\leq t < T_{i+1},
\end{aligned}}
\end{equation}
Hence, at the end of the $i$-th mode, the state satisfies $
	\|x_{T_{i+1}}\| \leq \kappa (1 - \frac{\alpha}{2})^{\tau} \|x_{T_{i}}\| +  \frac{2  \kappa}{\alpha} {(b_me_m+w_m)}.$
Denote the maximum of the state norm during $T_i \leq t \leq T_{i+1}$ as $x_{\text{max},i}$. Then, 
\begin{equation}\label{equ:beta}
\|x_{\text{max},i}\|  \leq {\kappa}\|x_{T_i}\| + {c_2}. 
\end{equation}

For the sake of simplicity, define $c_1 = \kappa (1 - \frac{\alpha}{2})^{\tau}$ and $c_2= \frac{2\kappa}{\alpha} {(b_me_m+w_m)}$. Then, it follows from \eqref{22} that
\begin{align*}
	\|x_{T_{i}}\| \leq c_1^i \|x_{T_{0}}\| +  {\frac{c_2}{1-c_1}},
\end{align*}
and \eqref{equ:beta} further leads to
$$
\|x_{\text{max},i}\| 
\leq {\kappa} c_1^i \|x_{T_0}\| + c_2+ \frac{{\kappa}c_2}{1-c_1}.
$$
{Finally, we prove the cumulative frozen-time optimality-gap bound. For each mode $i$ and each $t\in[T_i,T_{i+1})$, the descent inequality above gives}
\[
{
C_i(K_{t+1})-C_i(K_t)
\leq
-\frac{\eta}{2\bar\mu}\big(C_i(K_t)-C_i^*\big)
+\eta\bar p_6\bar p_3\epsilon_{\rm sw}.}
\]
{Let $T_i^T:=\min\{T_{i+1},T\}$. Summing this inequality over all active intervals up to the horizon $T$ yields}
\[
{\begin{aligned}
&\sum_{i=0}^{M_T}\sum_{t=T_i}^{T_i^T-1}
\big(C_i(K_t)-C_i^*\big)\\
&\leq
\frac{2\bar\mu}{\eta}
\sum_{i=0}^{M_T}
\big(C_i(K_{T_i})-C_i(K_{T_i^T})\big)\\
&\quad
+2\bar\mu\bar p_6\bar p_3(T-T_0)\epsilon_{\rm sw}.
\end{aligned}}
\]
{The sum over modes telescopes up to the objective jumps at switching instants. Using nonnegativity of the terminal cost and the uniform perturbation bound
$|C_i(K)-C_{i-1}(K)|\leq \bar p_{10}^*\delta$ on the bounded-cost set, it follows that}
\[
{
\sum_{i=0}^{M_T}
\big(C_i(K_{T_i})-C_i(K_{T_i^T})\big)
\leq
C_0(K_{T_0})+M_T\bar p_{10}^*\delta .}
\]
{Consequently,}
\[
{\begin{aligned}
&\sum_{i=0}^{M_T}\sum_{t=T_i}^{T_i^T-1}
\big(C_i(K_t)-C_i^*\big)\\
&\leq
\frac{2\bar\mu C_0(K_{T_0})}{\eta}
+\frac{2\bar\mu\bar p_{10}^*M_T\delta}{\eta}
+2\bar\mu\bar p_6\bar p_3(T-T_0)\epsilon_{\rm sw},
\end{aligned}}
\]
{Dividing both sides by $T-T_0$ yields \eqref{equ:regret_sw} after renaming constants.}
The proof {is complete}.

\bibliographystyle{plain}
\bibliography{mybibfile}

\end{document}